\documentclass[12pt]{article}
\usepackage{amsmath}
\usepackage{amssymb}
\usepackage{amsthm}
\usepackage{epsfig}
\usepackage{graphics}

\newtheorem{prop}{Proposition}[section]
\newtheorem{theo}[prop]{Theorem}
\newtheorem{coro}[prop]{Corollary}

\theoremstyle{definition}
\newtheorem{lemm}[prop]{Lemma}
\newtheorem{defi}[prop]{Definition}

\theoremstyle{remark}

\newtheorem{conj}[prop]{Conjecture}

\newtheorem{rema}[prop]{Remark}
\newtheorem{exam}[prop]{Example}
\newtheorem{obse}[prop]{Observation}

\newcommand{\bA}{\mathbb A}
\newcommand{\bC}{\mathbb C}
\newcommand{\bG}{\mathbb G}
\newcommand{\bN}{\mathbb N}
\newcommand{\bP}{\mathbb P}
\newcommand{\bQ}{\mathbb Q}
\newcommand{\bR}{\mathbb R}
\newcommand{\bZ}{\mathbb Z}

\newcommand{\cB}{\mathcal B}
\newcommand{\cC}{\mathcal C}
\newcommand{\cD}{\mathcal D}
\newcommand{\cI}{\mathcal I}
\newcommand{\cL}{\mathcal L}
\newcommand{\cM}{\mathcal M}
\newcommand{\cO}{\mathcal O}
\newcommand{\cP}{\mathcal P}
\newcommand{\cR}{\mathcal R}
\newcommand{\cT}{\mathcal T}
\newcommand{\cU}{\mathcal U}
\newcommand{\cX}{\mathcal X}
\newcommand{\cY}{\mathcal Y}
\newcommand{\cZ}{\mathcal Z}

\newcommand{\ocM}{\overline{\mathcal M}}

\newcommand{\hFb}{{\widehat F}_b}
\newcommand{\hOb}{{\widehat \cO}_{B,b}}

\newcommand{\fm}{{\mathfrak m}}

\newcommand{\ra}{\rightarrow}
\newcommand{\lra}{\longrightarrow}

\newcommand{\Hilb}{\mathcal{H}ilb}
\newcommand{\Hom}{\mathcal{H}om}
\newcommand{\Ext}{\mathcal{E}xt}

\newcommand{\Spec}{\mathrm{Spec}}
\newcommand{\Proj}{\mathrm{Proj}}
\newcommand{\Sect}{\mathrm{Sect}}
\newcommand{\oSect}{\overline{\mathrm{Sect}}}

\newcommand{\Pic}{\mathrm{Pic}}
\newcommand{\Bl}{\mathrm{Bl}}
\newcommand{\ev}{\mathrm{ev}}

\begin{document}

\title{Weak approximation and rationally connected varieties over function 
fields of curves}

\author{Brendan Hassett
\thanks{Supported in part by NSF Grants 0134259, 0554491, and 0901645.}}
\maketitle

\section*{Introduction}

This paper surveys recent work on weak approximation for varieties over complex function fields.  
It also touches on the geometric theory of rationally connected varieties and stable maps.  

Weak approximation has been studied extensively in the context of number theory, quadratic forms, and linear algebraic
groups.  Early examples include work of Kneser \cite{Kneser2} and Harder \cite{Harder} on linear algebraic groups over various fields.
In the 1980's, attention shifted to rational surfaces over number fields and cohomological obstructions to weak approximation.
Significant results were obtained by Colliot-Th\'el\`ene, Sansuc, 
Swinnerton-Dyer, Skorobogatov, Salberger, Harari, and others.
We refer the reader to \cite{Harari} for an excellent survey of the state of
this area in 2002.  

With the development of the theory of rationally connected varieties,
weak approximation over function fields of {\em complex} curves became a focus of research.  Already in 1992, Koll\'ar-Miyaoka-Mori
\cite{KMM} showed that rationally connected varieties over such fields enjoy remarkable approximation properties, assuming
they admit rational points.  In 2001,
Graber-Harris-Starr \cite{GHS} showed these rational points exist, which opened the door to a more systematic study of their properties.  

This paper is organized as follows:  Section~\ref{sect:elements} reviews the basic definitions, presenting them in a form useful for
our purposes.  In Section~\ref{sect:general}, we present results valid for general rationally connected varieties, as well as 
key constructions and deformation-theoretic tools.   We turn to special classes of varieties in Section~\ref{sect:cases}, including
rational surfaces and hypersurfaces with mild singularities at places of bad reduction.  Section~\ref{sect:RSC} addresses
a large class of varieties where weak approximation is known, the rationally simply connected varieties \cite{dJS}.  
We raise some questions for further study in Section~\ref{sect:questions}.  
The Appendix presents basic facts on stable maps used throughout the volume.  

\

\noindent {\bf Acknowledgments:}  
This survey is based on lectures given in May 2008 
at the Session Etats de la Recherche ``Vari\'et\'es rationnellement connexes : aspects g\'eom\'etriques et arithm\'etiques'',
sponsored by the Soci\'et\'e math\'ematique de France, the
Institut de Recherche Math\'ematique Avanc\'ee (IRMA), the Universit\'e Louis Pasteur Strasbourg, and the Centre National de la Recherche Scientifique (CNRS).  
I am grateful to the organizers, Jean-Louis Colliot-Th\'el\`ene, Olivier Debarre, and Andreas H\"oring, for their support and encouragement for me to 
write up these notes.  

My personal research contributions in this area are in collaboration with Yuri Tschinkel, who made helpful suggestions on the content and
presentation.  Jason Starr made very important contributions to the arguments presented in Section~\ref{sect:RSC} linking weak approximation
to rational simple connectedness.  I benefited from comments from the other speakers at the school, Laurent Bonavero and Olivier Wittenberg,
and from the participants, especially Amanda Knecht and Chenyang Xu.  I am grateful to Colliot-Th\'el\`ene for his constructive comments on drafts 
of this paper.

\section{Elements of weak approximation}
\label{sect:elements}

\paragraph{Notation} 

Throughout, a {\em variety} over a field $L$ designates a separated geometrically
integral scheme of finite type over $L$;  its {\em generic point} is the unique
point corresponding to its function field.  A {\em general point} of a variety is a closed point 
chosen from the complement of an unspecified Zariski closed proper subset.  

Let $k$ be an algebraically closed field of characteristic zero and $B$ a smooth
projective curve over $k$, with function field $F=k(B)$.  

Let $X$ be a smooth projective variety over $F$.
A {\em model} of $X$ is a flat proper morphism
$\pi:\cX \ra B$
with generic fiber $X$.  Usually, $\cX$ is a scheme projective over $B$, but
there are situations where we should take it to be an algebraic space proper over $B$.
For each $b\in B$, let $\cX_b=\pi^{-1}(b)$ 
denote the fiber over $b$.  Once we have chosen a concrete embedding
$X \subset \bP^N$, the properness of the  Hilbert scheme yields a natural
model.  The model is {\em regular} if the total space
$\cX$ is nonsingular;  this can always be achieved via resolution of singularities.  

\paragraph{Elementary properties of sections}
Recall that a {\em section} of $\pi$ is a morphism
$s:B\ra \cX$
such that $\pi \circ s:B \ra B$ is the identity.
By the valuative criterion of properness, we have
$$\left\{  \text{ sections }s:B \ra \cX 
				\text{ of } \pi 
\right\}
	\Leftrightarrow
\left\{
\text{ rational points } 
x\in X(F)	
\right\}.
$$
Assume $\cX$ is regular and write
\begin{equation} \label{eq:smooth}
\begin{array}{rcl}
\cX^{sm}&=&\{x \in \cX : \pi \text{ is smooth at } x \} \\
               &=& \{x \in \cX: \cX_b \text{ is smooth at }x, b=\pi(x) \} \subset \cX.
\end{array}
\end{equation}
Then each section $s:B\ra \cX$ is necessarily contained in $\cX^{sm}$.
The proof of this assertion is basic calculus:  Since $\pi \circ s$ is
the identity the derivative $d(\pi \circ s)$ is as well,
which means that $d\pi$ is surjective and 
$$\dim \mathrm{ker}(d\pi_{s(b)})=\dim \cX_b$$
for each $b\in B$.  Thus $\cX_b$ is smooth at $s(b)$.  

\paragraph{Formulating weak approximation}
For each $b\in B$, let $\hOb$ denote the completion of the local ring 
$\cO_{B,b}$ at the maximal ideal $\fm_{B,b}$, and $\hFb$ the completion of $F=k(B)$
at $b$, i.e., the quotient field of $\hOb$.  
Consider the {\em ad\`eles} over $F$
$$\bA_F={\prod_{b \in B}}' \hFb,$$
i.e., the restricted product over all the places of $B$.
The restricted product means that all but finitely many of the factors are in 
$\hOb$.   There are {\em two} natural topologies one could consider:
The ordinary product topology and the restricted product
topology, with basis consisting of products of open sets
$\prod_{b \in B}  U_b$, where $U_b=\hOb$ for all but finitely many $b$.
Using the natural inclusions $F\subset \hFb$, we may regard $F\subset \bA_F$.

Given a variety $X$ over $F$, the adelic points $X(\bA_F)$ inherit both topologies
from $\bA_F$.    

\begin{defi}
A variety $X$ over $F$ satisfies {\em weak approximation} (resp.~{\em strong 
approximation}) if 
$$X(F) \subset X(\bA_F)$$
is dense in the ordinary product (resp.~restricted product) topology.    
\end{defi}

For {\em proper} varieties $X$, the distinction between weak and strong
approximation is irrelevant.  This will be clear after we 
analyze the definitions in this case.       

\paragraph{Unwinding the definition}

Assume $X$ is smooth and proper and fix a model $\pi:\cX \ra B$.  Both topologies on   
$X(\bA_F)$ have the following basis:  Consider data
$$J=(N;b_1,\ldots,b_r;\hat{s}_1,\ldots,\hat{s}_r),$$
consisting of a nonnegative integer $N$, distinct places $b_1,\ldots,b_r\in B$,
and points $\hat{s}_i \in X(\widehat{F}_{b_i})$ for $i=1,\ldots,r$.
Since $\pi: \cX\ra B$ is proper, we may interpret $\hat{s}_i$ as a section of
the restriction
$$  \pi|\widehat{B}_{b_i}:\cX \times_B \widehat{B}_{b_i} \ra \widehat{B}_{b_i},
		\quad \widehat{B}_{b_i}=\Spec( \widehat{\cO}_{B,b_i}),$$
to the completion of $B$ at $b_i$.  Since we can freely
clear denominators, insisting that the points are integral at almost all
places is not a restriction.  Thus our basic open sets are
$$U_J=\{ t \in X(\bA_F): t\equiv \hat{s}_i  \pmod{\fm^{N+1}_{B,b_i}}\},$$
i.e., sections with Taylor series at $b_1,\ldots,b_r$ prescribed to 
order $N$.   

Now suppose in addition that $\pi:\cX \ra B$ is a {\em regular} model,
so that sections automatically factor through $\cX^{sm} \subset \cX$
(see (\ref{eq:smooth}) above).  Note that $\hat{s}_i(b_i)\in \cX_{b_i}$ is a smooth
point, by the same calculus argument we used to show sections factor
through $\cX^{sm}$.  Conversely,
Hensel's Lemma (or the $\fm$-adic version of Newton's method, cf.~\cite[p.14]{Se})
implies that each section $\hat{s}^N_{i}$ of 
$$\cX^{sm} \times_B \Spec(\cO_{B,b_i}/\fm^{N+1}_{B,b_i})
\ra \Spec(\cO_{B,b_i}/\fm^{N+1}_{B,b_i})
$$
can be extended to a section $\hat{s}_i$ of $\pi|\widehat{B}_{b_i}.$
Thus we can recast our data as a collection of {\em jet data}
\begin{equation} \label{eq:jetdata}
J=(N;b_1,\ldots,b_r;\hat{s}^N_1,\ldots,\hat{s}^N_r),
\end{equation}
where the $\hat{s}^N_i$ are $N$-jets of sections of $\cX^{sm} \ra B$
at $b_i$.  
\begin{figure}
\begin{center}
\includegraphics{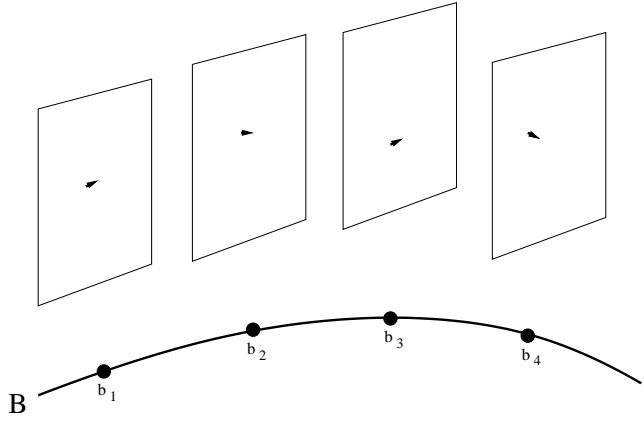}
\end{center}
\caption{Jet data for sections}
\end{figure}

To summarize:
\begin{obse} \label{obse1}
Let $X$ be a smooth proper variety over $F$.  To establish weak approximation for $X$,
it suffices to show, for {\em one} regular model $\cX \ra B$, that for each
collection of jet data (\ref{eq:jetdata}) there exists a section $s:B\ra \cX$
with $s\equiv \hat{s}^N_i \pmod{\fm^{N+1}_{B,b_i}}.$
\end{obse}

\paragraph{Iterated blow-ups arising from formal sections}
Let $X$ be a smooth proper variety over $F$ of dimension $d$,
with regular model $\pi:\cX \ra B$.  
Fix a point $b\in B$ and a formal section
$$\hat{s}:\widehat{B}_b \ra \cX\times_B \widehat{B}_b.$$
This is equivalent to a point of $X(\hFb)$.

We define a sequence of new models
$$\cX^N \ra \cX^{N-1} \ra \cdots \ra \cX^1 \ra \cX^0=\cX$$
inductively as follows:  
\begin{enumerate}
\item{Set $\cX^1=\Bl_{\hat{s}(b)}(\cX^0)$ and let 
$$\hat{t}^1:\widehat{B}_b \ra \cX^1 \times_B \widehat{B}_b$$
denote the induced section, which exists by applying the
valuative criterion to $\cX^1 \ra B$.}
\item{Set $\cX^2=\Bl_{\hat{t}^1(b)}(\cX^1)$ and 
$$\hat{t}^2:\widehat{B}_b \ra \cX^2 \times_B \widehat{B}_b$$
the induced section.}
\item[ ]{\hspace{2.5in} \vdots}
\item[N.]{Set $\cX^N=\Bl_{\hat{t}^{N-1}(b)}(\cX^{N-1})$ and 
$$\hat{t}^N:\widehat{B}_b \ra \cX^N \times_B \widehat{B}_b$$
the induced section.}
\end{enumerate}
In other words, we blow up successively along the proper transforms
of the formal section over $b$.  
Note that $\cX^1$ depends only on  $\hat{s}(b)=\hat{s}^{0}=\hat{s}\pmod{\fm_{B,b}}$,
and in general, 
$\cX^N$ depends only on the jet datum $\hat{s}^{N-1}=\hat{s}\pmod{\fm^N_{B,b}}$.  
The fiber over $b$ can be expressed
\begin{equation} \label{eq:iterfiber}
\cX^N_b=\Bl_{\hat{s}(b)}(\cX_b) \cup \Bl_{\hat{t}^{1}(b)}(\bP^d) \cup \cdots \cup \Bl_{\hat{t}^{N-1}(b)}(\bP^d) \cup \bP^d,
\end{equation}
i.e., as a chain with the proper transform of $\cX_b$ at one end, the $N$th exceptional divisor at the other end,
and blow-ups of the intermediate exceptional divisors in between.  
Observe that $\hat{t}^N(b) \in \bP^d$, the $N$th exceptional divisor.  
\begin{figure}
\begin{center}
\includegraphics{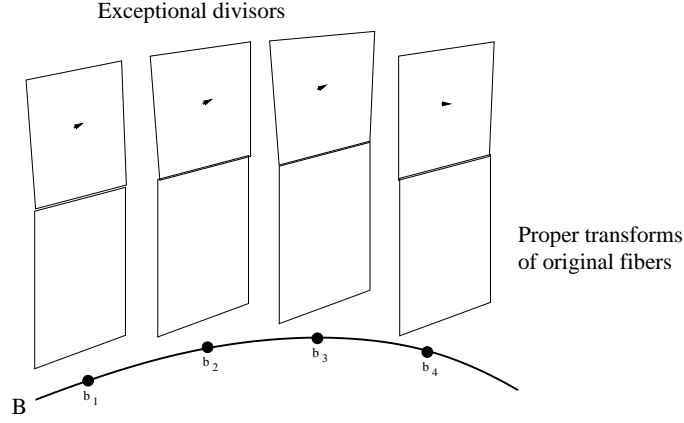}
\end{center}
\caption{Iterated blow-up construction}
\end{figure}

\begin{defi} \label{defi:iterate}
Let $X$ be a smooth proper variety over $F$ with regular model $\pi:\cX \ra B$, and
$$
J=(N;b_1,\ldots,b_r;\hat{s}^N_1,\ldots,\hat{s}^N_r)
$$
a collection of jet data as in (\ref{eq:jetdata}).  The {\em iterated blow-up}
associated to $J$ 
$$\beta^J:\cX^J \ra \cX$$
is obtained by blowing up $N$ times along over each point $b_j$.  
For each $i=1,\ldots,r$, let $x^J_i \in \cX^J_{b_i}$ denote the evaluation of the 
formal section defining the blow-up over $b_i$ at the closed point.
\end{defi}

\begin{prop} \label{prop:iterate}
Retaining the notation of Definition~\ref{defi:iterate}, we have a natural bijection
$$\left\{ \begin{array}{c} \text{sections }s:B\ra \cX \\
	   \text{with jet data } $J$ \end{array} \right\} 
\Leftrightarrow
\left\{	\begin{array}{c} \text{sections }s^J:B \ra \cX^J \\
	   \text{with } s^J(b_i)=x^J_i \end{array}  \right\}.$$
The direction $\Rightarrow$ is induced by taking the proper transform of $s$ in $\cX_J$;  the
direction $\Leftarrow$ is given by setting $s=\beta^J \circ s^J$.  
\end{prop}

In other words, we can interpret weak approximation on $\cX$ in terms of finding 
sections with prescribed values over the various iterated blow-ups of $\cX$.  
To summarize:
\begin{obse} \label{obse2}
Let $X$ be a smooth proper variety over $F$.  To establish weak approximation for $X$,
it suffices to show, for {\em each} regular model $\cX \ra B$, distinct places
$b_1,\ldots,b_r \in B$, and smooth points $x_i \in \cX^{sm}_{b_i},i=1,\ldots,r$,  
there exists a section $s:B\ra \cX$ with $s(b_i)=x_i$ for each $i$.
\end{obse}

\paragraph{Weak approximation and birational models}
Weak approximation is a birational property (see \cite[\S 2.1]{Kneser2}):
\begin{theo} \label{theo:birat}
Let $X_1$ and $X_2$ be smooth varieties over $F$.  Assume they are 
birational over $F$.  Then $X_1$ satisfies weak approximation if and only
if $X_2$ satisfies weak approximation.   
\end{theo}
\begin{coro} \label{coro:rational}
Let $X$ be a smooth variety, rational over $F$.  Then $X$
satisfies weak approximation.  
\end{coro}
\begin{proof}
Since $X_1$ and $X_2$ are birational, they share a common Zariski open dense subset
$$X_1 \supset U \subset X_2.$$
We claim that $U(\bA_F)$ is dense in $X_1(\bA_F)$ and $X_2(\bA_F)$
in the ordinary product topology, or equivalently, 
$$X_i(\bA_F) \setminus U(\bA_F), \quad i=1,2,$$
has trivial interior.  Given a point 
$x \in  X_1(\bA_F)\setminus U(\bA_F)$, we claim there exists a point
$u \in U(\bA_F)$ that approximates $x_1$ (in the direct product
topology) arbitrarily closely.  Precisely, consider distinct
places $b_1,\ldots,b_r \in B$ and the corresponding projections
$$x_{b_j} \in X(\widehat{F}_{b_j}).$$
It suffices to show that for each $N\ge 0$, there exist 
$$u_{b_j} \in U(\widehat{F}_{b_j})$$
such that $u_{b_j}\equiv x_{b_j} \pmod{\fm_{B,b_j}^{N+1}}.$

We take $u_{b_j}$ to be an $N$th order approximation of $x_{b_j}$,
chosen in such a way that it does {\em not} lie on $X_1\setminus U$ (regarded
as a subscheme with the reduced induced scheme structure).  
Standard $\fm$-adic approximation techniques (cf.~\cite[p. 14]{Se}) allow us to exhibit such points.  
\end{proof}

\begin{rema}
In geometric terms, nonempty open subsets of $X_1(\widehat{F}_{b_j})$ are Zariski dense in $X$.
We can produce points intersecting $X_1\setminus U$ over $b_j$ with arbitrarily high order.  
\end{rema}

\paragraph{Places of good reduction and construction of models}
\begin{defi}
Let $X$ be a smooth proper variety over $F$ with regular model $\cX\ra B$.
The model has {\em good reduction} at $b\in B$ if $\cX_b$
is smooth;  otherwise, it has {\em bad reduction} at $b$.  
A place $b\in B$ is of {\em good reduction} if there exists a regular
model $\pi:\cX \ra B$ with good reduction at $b$, and of {\em bad
reduction} otherwise.  
\end{defi}

It is clear that each model has only finitely many places of bad reduction.  
A gluing argument gives:
\begin{prop} \cite[Prop. 7]{HT06}
Let $X$ be a smooth proper variety over $F$.  Then there exists an algebraic space
$\cX \ra B$ that is a regular model for $X$ and is smooth over each place of good reduction.
\end{prop}
Such models are called {\em good models}.  

\paragraph{An example: pencils of cubic surfaces}
We really do need to consider algebraic spaces to get good models:

Consider the $\bP^{19}$ parametrizing all cubic surfaces in $\bP^3$,
with discriminant hypersurface $D\subset \bP^{19}$.  The fundamental group of
$U=\bP^{19}-D$ acts on the cohomology lattice via the full monodromy
representation.
The monodromy group is the Weyl group $W(E_6)$, acting on the
primitive cohomology lattice of the cubic surface via the standard
representation generated by reflections in the simple roots \cite{Har79}.  
(These simple roots are differences of disjoint lines on the cubic surface.)   
There is a normal simple subgroup $H\subset W(E_6)$ of index two
and order $25920$, corresponding
to the elements acting on the lattice with determinant one.

Take a general line in $\bP^{19}$, i.e., one intersecting $D$ transversally.
This family may be represented by the bihomogeneous equation
$$\cX=\{sF(w,x,y,z)+tG(w,x,y,z)=0\} \subset \bP^1 \times \bP^3, \quad       \deg(F)=\deg(G)=3.$$
The resulting cubic surface fibration 
$\pi:\cX\ra \bP^1$ has the
following properties:
\begin{enumerate}
\item $\cX$
is nonsingular and each singular fiber $\cX_b$
has a single ordinary double point, i.e., an isolated singularity
with smooth projective tangent cone; \'etale locally we have
$$\cX_b\sim \{x_0x_1+x_2^2=t=0\} \subset \{x_0x_1+x_2^2=t \} \sim \cX.$$
\item
the local monodromy near each singular fiber
is a reflection in a simple root, i.e., the vanishing cycle
        for the ordinary double point;
\item $\pi$ has $32$ singular fibers;
\item the monodromy representation is the full group $W(E_6)$.
\end{enumerate}
The first assertion is just a restatement of the generality assumption;
the second is the Picard-Lefschetz monodromy formula.
The third property follows from a straightforward computation with
Euler characteristics:  $\cX$ is the blow-up of $\bP^3$
along the complete intersection $F=G=0$.  The last property follows from
the Lefschetz hyperplane theorem for fundamental groups, applied to the
open variety $U$ \cite[p. 150]{GM}.

Let $\{p_1,....,p_{32}\}$ be the intersection of $D$ with our general line,
i.e., the discriminant of $\pi$.  Let $B$ denote the double
cover of $\bP^1$ at these points and  
$$\pi':\cX':=\cX\times_{\bP^1} B \ra B$$
the pullback to $B$.  The local monodromy near the singular fibers of $\pi'$ 
is now trivial;  at the singular point of each such fiber,
$\cX'$ has an ordinary threefold double point, \'etale locally
isomorphic to $\{x_0x_1+x_2^2=u^2\}$.  

The monodromy representation is an index two subgroup of $W(E_6)$,
which does not contain the Picard-Lefschetz reflections
associated to the degenerate fibers;  this must be
the simple group $H$.  The restriction
of the standard reflection representation of $W(E_6)$ to $H$ is
still irreducible.  There are no primitive classes fixed under $H$ and
$\Pic(\cX'/B)$ is still generated by the anti-canonical class.

The family $\pi'$ admits a simultaneous resolution 
$$\begin{array}{rcl}
\cY& \ra & \cX' \\
  &   & \downarrow \\
  &    & B,
\end{array}
$$
as the local monodromy near each singular fibers is trivial \cite{Bries}.
Concretely, one takes a small resolution of each of the $32$
ordinary singularities of $\cX'$.  This entails replacing each singularity with
a $\bP^1$;  \'etale locally, small resolutions of $\{x_0x_1+x_2^2=u^2 \}$
can be obtained by blowing up either of the planes
 $\{x_0=x_2-u=0\}$ or $\{x_1=x_2-u=0\}$.  The irreducibility of the monodromy
representation means $\cY$ is not even locally projective over $B$.  It follows
that $\cY$ exists as an algebraic space but not as a scheme.

\section{Results for general rationally connected varieties}
\label{sect:general}
\paragraph{Basic properties of rationally connected varieties}
There are numerous basic references for rationally connected 
varieties, e.g., Koll\'ar's book \cite[IV.3]{Kolbook}, Debarre's book
\cite[ch.~4]{Debbook}, and Bonavero's contribution to this volume.  
And we should mention the original papers \cite{KMM} and \cite{Cam}.  

One general point on terminology:  A {\em rational curve} on a variety
$Y$ is the image of a non-constant morphism $f:\bP^1 \ra Y$.  In particular,
rational curves are always proper, even when $Y$ is not proper.

\begin{defi}
Let $Y$ be a smooth variety.  It is {\em rationally connected}
(resp.~{\em rationally chain connected}) if there exists a 
proper flat morphism $Z\ra W$ over a variety $W$, whose fibers
are irreducible (resp.~connected) curves of genus zero, and
a morphism
$\phi:Z\ra Y$
such that the induced morphism
$$\phi^2: Z\times_W Z \ra Y\times Y$$
is dominant.
\end{defi} 
Roughly, $Y$ is rationally connected if
there exists a rational curve $f:\bP^1 \ra Y$ through 
the generic pair of points $(y_1,y_2) \in Y\times Y$.

\begin{defi}
Let $Y$ be a smooth variety.  A non-constant morphism
$f:\bP^1 \ra Y$ is {\em free} (resp.~{\em very free}) if
we have a decomposition
$$f^*T_Y \simeq \oplus_{i=1}^{\dim(Y)} \cO_{\bP^1}(a_i)$$
with each $a_i\ge 0$ (resp.~$a_i >0$).
\end{defi}

Trivially, every rationally connected variety is rationally
chain connected and each variety admitting a very free curve
admits a free curve.  Slightly less trivial is that every
variety admitting a very free curve is in fact rationally
connected.  For smooth varieties $Y$ over fields of characteristic
zero, we can prove converses to these statements.
Generic smoothness 
implies that the general rational curve in a family 
dominating $Y \times Y$ is very free.     
And smooth rationally chain connected
varieties are in fact rationally connected.  

\paragraph{Elements of deformation theory}
We recall a basic result we shall use frequently.

Let $Y$ be a smooth variety.  
Consider a map $\{f:(C,x_1,\ldots,x_r)\ra Y\}$ defined on
a nodal proper connected curve
with distinct marked smooth points, with $f$ non-constant and 
unramified at the nodes and the marked points.  
In particular, $f$ does not contract any irreducible components of $C$,
so $\{f:C \ra Y \}$ and $\{f:(C,x_1,\ldots,x_r) \ra Y \}$ are both stable maps.  
(See the Appendix for background on stable maps.)  

Let $N_f$ denote the {\em normal sheaf} of $f$, defined (see \cite[p. 61]{GHS})
as the unique non-vanishing cohomology group of
$$\bR \Hom_{\cO_C} (f^*\Omega^1_Y \stackrel{df^{\dag}}{\ra} \Omega^1_C,\cO_C),$$
which fits into an exact sequence
$$0 \ra \Hom_{\cO_C}(\mathrm{ker}(df^{\dag}),\cO_C) \ra N_f
\ra \Ext^1_{\cO_C}(\mathrm{coker}(df^{\dag}),\cO_C) \ra 0.$$
When $C$ is smooth, $\Omega^1_C$ is invertible and we obtain 
$$0 \ra T_C \ra f^*T_Y \ra N_f \ra 0.$$
Generally, the tangent space and obstruction space to the moduli stack of
stable maps at
$\{f:C\ra Y \}$ are $\Gamma(C,N_f)$ and $H^1(C,N_f)$ respectively.  
Sometimes, we will abuse notation and write $N_{C/Y}$ for $N_f$.  

Let $\cI_{x_1,\ldots,x_r}$ denote the 
ideal sheaf for the marked points.  
Consider the substack of the stable map space consisting of
$\{f':(C',x'_1,\ldots,x'_r)\ra Y \}$ such
that $f'(x'_i)=f(x_i)$ for each $i$.  Its tangent space at
$\{f:(C,x_1,\ldots,x_r) \ra Y\}$ equals
$\Gamma(C,N_f\otimes \cI_{x_1,\ldots,x_r});$
the obstruction space is
$H^1(C,N_f\otimes \cI_{x_1,\ldots,x_r}).$

\begin{lemm} \cite[\S 2]{GHS} \label{lemm:smooth}
Suppose $N_f\otimes \cI_{x_1,\ldots,x_r}$ has no higher cohomology and
is generated at the nodes by global sections.   Then $f:C \ra Y$
admits a smoothing $f'$ with $f'(x_i)=f(x_i)$ for each $i$.  
\end{lemm}

Lemma~\ref{lemm:smooth} is typically used to show that stable
maps admit smoothings provided they have enough free or very
free curves among their irreducible components.  It
can be applied to establish:
\begin{prop} \cite[IV.3.9.4]{Kolbook}  \label{prop:thrupts}
Let $Y$ be a smooth rationally connected 
variety over $k$, an algebraically closed
field of characteristic zero.    

There exists a unique maximal nonempty open subset $Y^{\circ} \subset Y$
such that, for any finite 
collection of points $y_1,\ldots,y_m \in Y^{\circ}$, there
exists a morphism $f:\bP^1 \ra Y^{\circ}$ 
with image containing the points.  Any rational curve meeting
$Y^{\circ}$ is contained in $Y^{\circ}$.  

If $Y$ is proper $Y^{\circ}=Y$.
\end{prop}
Currently, no example is known where $Y^{\circ} \subsetneq Y$.
Actually, we shall require a slight variation on this result:
\begin{prop} \cite[2.2]{DebBour}  \label{prop:thrujets}
Given any collection of points $y_1,\ldots,y_m \in Y^{\circ}$ and non-trivial
tangent vectors $v_j \in T_{y_j}Y^{\circ}$, there exists a 
morphism $f:\bP^1 \ra Y^{\circ}$ and points 
$p_1,\ldots,p_m \in \bP^1$ such that $f(p_j)=x_j$ and 
$\mathrm{image}(df_{p_j})=\mathrm{span}(v_j)$ for $j=1,\ldots,m$.
\end{prop}  
In fact, we can produce a rational curve through an arbitrary
curvilinear subscheme of $Y^{\circ}$;  see \cite[Prop. 13]{HT08}.  
(By definition, a zero-dimensional scheme is {\em curvilinear}
if it can be embedded into a smooth curve.)
The existence of rational curves with prescribed
tangencies can be deduced from Proposition~\ref{prop:thrupts}.  
This foreshadows our weak approximation argument.

\begin{proof}
Fix $y\in Y^{\circ}$
and $0\neq v \in T_yY^{\circ}$, and suppose we are given $f:\bP^1 \ra Y^{\circ}$
with $f(0)=y$.  Since we may choose $f$ to pass through a large of number of
prescribed points in addition to $y$, we may assume that $f$ is very free.
Let 
$$\tilde{f}:\bP^1 \ra \widetilde{Y}=\Bl_y(Y)$$
denote the lift to the blow-up of $Y$ at $y$;  let 
$$E=\bP(T_yY^{\circ})\simeq \bP^{\dim(Y)}\subset \widetilde{Y}$$ 
denote the exceptional divisor and $[v] \in E$ the point corresponding to $v$.  
Since
$$\tilde{f}^*T_{\widetilde{Y}}=f^*T_Y\otimes \cI_0\simeq f^*T_Y(-1),$$
the curve $\tilde{f}$ is at least free.  Thus a general deformation $g_t$
of $\tilde{f}$ meets $E$ at a general point.  

Consider {\em two} such deformations $g_1$ and $g_2$, 
meeting $E$ at $w_1$ and $w_2$ respectively.  Choose these in such a
way that the line $\ell\subset E$ joining $w_1$ and $w_2$ contains
$[v]$.  Consider the reducible curve
$$C=g_1(\bP^1) \cup \ell \cup g_2(\bP^1);$$         
we analyze the normal bundle $F:=N_{C/\widetilde{Y}} \otimes \cI_{[v]}$, corresponding
to infinitesimal deformations of $C \subset \widetilde{Y}$ that
still contain $[v]$.   By Lemma~\ref{lemm:smooth}, it suffices to show this
is globally generated and has no higher cohomology.   
\begin{figure}
\begin{center}
\includegraphics{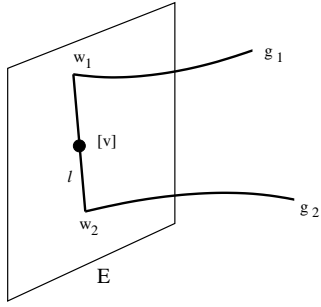}
\end{center}
\caption{The reducible curve $C$}
\end{figure}

We know that 
$F|g_i(\bP^1)$ is globally generated with no higher cohomology because $g_i$ is free.  
We compute $F|\ell$ using the exact sequences:
$$
\begin{array}{cccccccccc}
  & 	&	     &     &  0            &     &
        0               &     &  \\
  & 	&	     &     & \downarrow             &     &
        \downarrow               &     &  \\
0 & \ra & N_{\ell/E} & \ra & N_{\ell/\widetilde{Y}} & \ra & 
	N_{E/\widetilde{Y}}|\ell & \ra & 0 \\
  & 	&     ||     &     & \downarrow             &     &
        \downarrow               &     & 0 \\
0 & \ra & N_{\ell/E} & \ra & N_{C/\widetilde{Y}}|\ell &  \ra &
        Q			&   \ra & 0 \\	
  & 	&          &     & \downarrow             &     &
        \downarrow               &     &  \\
  & 	&          &     & \cO_{\{w_1,w_2\}}             &  =   &
        \cO_{\{w_1,w_2\}}               &     &  \\
& 	&          &     & \downarrow             &     &
        \downarrow               &     &  \\
& 	&          &     &  0	            &     &
        0               &     &  
\end{array}
$$
Since $\ell$ is a line and $E$ is exceptional, we know
$$N_{\ell/E} \simeq \cO_{\bP^1}(1)^{\oplus \dim(Y)-2}, \quad 
N_{E/\widetilde{Y}}|\ell \simeq  \cO_{\bP^1}(-1).$$
However, this negativity is overcome by the contribution of the nodes $w_1$
and $w_2$, which implies that $Q \simeq \cO_{\bP^1}(1)$.  
Thus $N_{C/\widetilde{Y}}|\ell \simeq \cO_{\bP^1}(1)^{\oplus \dim(Y)-1}$
and  $F|\ell \simeq \cO_{\bP^1}^{\oplus \dim(Y)-1}$.      
\end{proof}

\paragraph{Koll\'ar-Miyaoka-Mori and Graber-Harris-Starr Theorems}
The following theorem is a forerunner of weak approximation
results:
\begin{theo} \cite[2.13]{KMM} \label{theo:KMM}
Let $X$ be a smooth proper rationally connected variety
over $F=k(B)$ with model $\pi:\cX \ra B$.  Assume $X(F)\neq \emptyset$,
i.e., there exists a section $s:B\ra \cX$.  Given distinct places $b_1,\ldots,b_r$
of good reduction for $\cX$ and points $x_i \in \cX_{b_i}$ for $i=1,\ldots,r$,
there exists a section $s':B\ra \cX$ with $s'(b_i)=x_i$ for each $i$.
\end{theo}
We shall sketch the main ideas of the proof below.  
It took another ten years to remove the assumption on the existence
of the section:
\begin{theo}  \cite{GHS} \label{theo:GHS}
If $X$ is a smooth proper rationally connected variety over $F=k(B)$
then $X(F)\neq \emptyset$.  
\end{theo}
A good survey of this important result is \cite{StarrClay}.  

\paragraph{Weak approximation at places of good reduction}
\begin{theo}\cite{HT06} \label{theo:WAgood}
Let $X$ be a smooth proper rationally connected variety over $F=k(B)$.  
Then weak approximation holds at places of good reduction for $X$.   
\end{theo}
We sketch the ideas of the proof in the next two paragraphs.
Throughout, let $d=\dim(X)$ and $\cX \ra B$ a good model for $X$;  
fix distinct places $b_1,\ldots,b_r \in B$ such that 
$\cX_{b_1},\ldots,\cX_{b_r}$ are smooth.  Let $s:B\ra \cX$ denote
the section coming from the Graber-Harris-Starr theorem.  Our argument proceeds by
induction on $N$, the order of the jets we seek to approximate.    
  
\paragraph{The base case}
The base ($N=0$) case is just the Koll\'ar-Miyaoka-Mori theorem (Theorem~\ref{theo:KMM}), which we sketch
for completeness.  

We start with a preparation step, which is a key ingredient of the Graber-Harris-Starr
theorem.  Basically, we need to show that every rationally connected fibration admits
a `nice section':
\begin{prop} \cite[\S 2]{GHS} \label{prop:nicesection}
Let $\cX \ra B$ be as above, with section $s:B\ra \cX$.  Then there exists 
$s':B \ra \cX$ such that the normal bundle $N_{s'}$ is globally generated, with
vanishing higher cohomology.
\end{prop}
Here is the idea:  Let $U\subset B$ denote the places of good reductive.  Given
any finite collection $b'_1,\ldots,b'_q \in U$ and non-trivial vertical
tangent vectors 
$v_i \in T_{s(b'_i)}\cX_{b'_i}$, Proposition~\ref{prop:thrujets}
gives very free curves $f_i:\bP^1 \ra \cX_{b'_i}$ passing through 
$s(b'_i)$ with tangent $v_i$.  If we choose sufficiently many $b'_1,\ldots,
b'_q$ and appropriate tangent directions $v_i$, the union
$$C=s(B) \cup_{s(b'_1)} f_1(\bP^1) \cup \cdots \cup_{s(b'_q)} f_q(\bP^1)$$
has normal bundle that is globally generated and has no higher cohomology;
for details, consult the `First construction' and Lemma 2.5 of \cite{GHS}.
Deformation theory (cf. Lemma~\ref{lemm:smooth}) allows us to 
smooth $C$ to a section $s':B \ra \cX$, whose normal bundle remains globally
generated without higher cohomology.      

The argument
is now fairly straightforward:  Fix points $x_i \in \cX_{b_i}$ over
places of good reduction $b_1,\ldots,b_r \in B$.  Let $s':B\ra \cX$
be our nice section and $g_i:\bP^1 \ra \cX_{b_i}$ very free curves
such that $g_i(0)=s'(b_i)$ and $g_i(\infty)=x_i$.  The union 
$$C=s'(B) \cup_{s'(b_1)} g_1(\bP^1) \cup \cdots \cup_{s'(b_r)} g_r(\bP^1)$$
has normal bundle $N_{C/\cX}$ such that
$N_{C/\cX} \otimes \cI_{x_1,\ldots,x_r}$
is globally generated without higher cohomology.  (See \cite[II.7.5]{Kolbook}
for a detailed cohomology analysis.)  Lemma~\ref{lemm:smooth} allows us to 
smooth $C$ to a section $s'':B \ra \cX$ that still contains $x_1,\ldots,x_r$.  
\begin{figure}
\begin{center}
\includegraphics{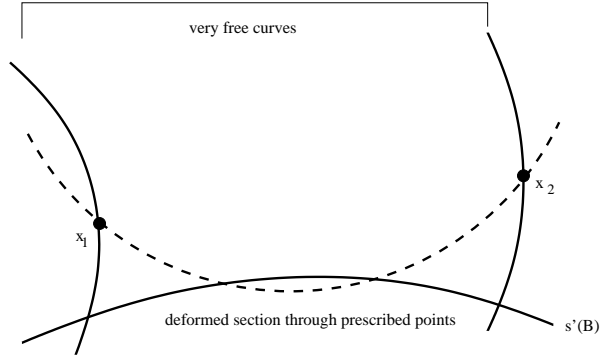}
\end{center}
\caption{Deformation for the Koll\'ar-Miyaoka-Mori theorem}
\end{figure}

\paragraph{The inductive step}
For simplicity, we describe the procedure in the special case where there is only one
place $b=b_1$.  Let $\hat{s}^N$ denote the $N$-jet of the formal section $\hat{s}$ we seek to
approximate and
$\beta:\cX^N \ra \cX$
the associated iterated blow-up construction with fiber 
$$\cX^N_b=\Bl_{\hat{s}(b)}(\cX_b) \cup \Bl_{\hat{t}^{1}(b)}(\bP^d) \cup \cdots \cup \Bl_{\hat{t}^{N-1}(b)}(\bP^d) \cup \bP^d.$$
(We retain the notation of Equation~(\ref{eq:iterfiber}).)  
Note that each $\Bl_{\hat{t}^{i}(b)}(\bP^d)$ is ruled by the proper transforms of the lines through
the blown-up point.  

Recall Proposition~\ref{prop:iterate}:  We seek sections $\sigma:B \ra \cX^N$ with $\sigma(b)=x^N$, where 
$$x^N \in \bP^d \subset \cX^N_b$$
corresponds to the jet $\hat{s}^N$.  
The inductive hypothesis gives a section
$\tau:B\ra \cX$
with  $\tau\equiv \hat{s}^N\pmod{\fm_{B,b}^{N}}$;  its proper transform $\tau^N:B \ra \cX^N$ satisfies
$y^N:=\tau^N(b) \in  \bP^d$.  
We may assume $\tau^N$ is nice, after another application of Proposition~\ref{prop:nicesection}.  
And there is nothing to prove unless $y^N \neq x^N$.  

We construct a chain of rational curves
$$T_0 \cup T_1 \cup \cdots \cup T_N 
$$
with
$$
T_0 \subset \mathrm{Bl}_{\hat{s}(b)}(\cX_b) \quad
T_i \subset \mathrm{Bl}_{\hat{t}^{i}(b)}(\bP^d), 0<i<N,  \quad
T_N \subset \bP^d,
$$
inductively as follows:
$T_N$ is the line joining $x^N$ and $y^N$, $T_{N-1}$ is the unique
ruling meeting $T_{N}$, etc., and $T_0$ is a very free curve meeting $T_1$.  
The reducible curve 
$$C=\tau^N(B) \cup T_N \cup T_{N-1} \cup \cdots \cup T_1 \cup T_0$$
admits a deformation to a section $\sigma:B\ra \cX^N$
still passing through $x^N$.  
\begin{figure}
\begin{center}
\includegraphics{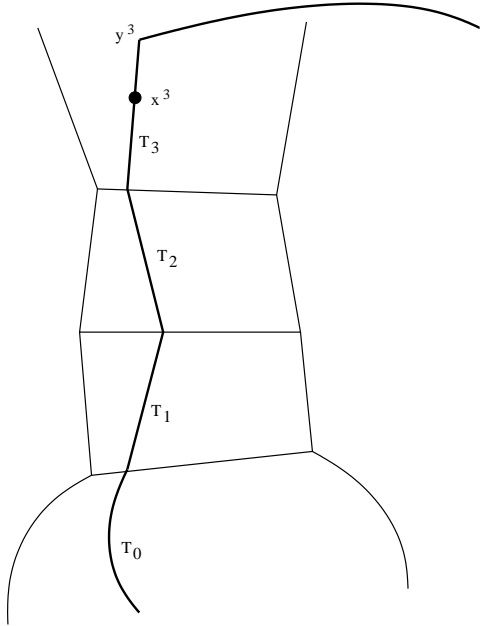}
\end{center}
\caption{The chain of rational curves for $N=3$}
\end{figure}

\paragraph{Direct generalizations}
\begin{defi}
A smooth (not necessarily proper)
variety $Y$ is {\em strongly rationally connected} if through
{\em any} 
point $y\in Y$ there passes a very free rational curve.
\end{defi}
We rephrase Proposition~\ref{prop:thrupts} as follows:
\begin{prop} \label{prop:SRC}
Let $Y'$ be a smooth rationally connected variety over an algebraically
closed field of characteristic zero.  Then there exists an open, nonempty subset
$Y^{\circ} \subset Y'$ that can be characterized as the maximal 
strongly rationally connected subset of $Y'$.  
\end{prop}
\begin{proof}
Let $Y^{\circ}$ denote the set produced from Proposition~\ref{prop:thrupts};
recall that any rational curve meeting $Y^{\circ}$ is contained in $Y^{\circ}$.

We have already seen (in the course of the proof of Proposition~\ref{prop:thrujets})
that $Y^{\circ}$ is strongly rationally connected:  The existence of curves
through arbitrary finite collections of points means that there are 
very free curves through any $y\in Y^{\circ}$.  

We show that $Y^{\circ}$ is maximal among strongly rationally connected subsets
of $Y'$.  If there exists such a subset containing $y\in Y'$ then there exists a 
very free curve through $y$, i.e., one joining $y$ to a general point of $Y'$.
In particular, such a curve necessarily meets $Y^{\circ}$ and thus is contained in
$Y^{\circ}$.  
\end{proof}

The proof of Theorem~\ref{theo:WAgood} generalizes directly as follows:
\begin{theo} \label{theo:WASRC}
Suppose $S\subset B$ is a finite set and $B^{\circ}=B\setminus S$.  
Let $\pi:\cY \ra B^{\circ}$ be a smooth morphism with strongly rationally
connected fibers.  Suppose that $\pi$ admits a section.
For any collection $J$ of jet data for $\cY$ over $B^{\circ}$ (cf. (\ref{eq:jetdata})),
there exists a section $s:B^{\circ} \ra \cY$ with those jet data.  
\end{theo}
This can be successfully applied to certain kinds of singular fibers, e.g.,
where $\cX \ra B$ is a regular model such that $\cX^{sm} \ra B$ has
strongly rationally connected fibers.  We shall offer specific examples
in Section~\ref{sect:cases} below.

\paragraph{A conjecture}
The Koll\'ar-Miyaoka-Mori theorem and the weak approximation results
sketched above motivate the following general assertion:
\begin{conj}[Weak approximation for rationally connected varieties]  \cite{HT06}
Let $X$ be a rationally connected variety over $k(B)$, the function
field of a curve $B$ over $k$, an algebraically closed field of
characteristic zero.  Then $X$ satisfies weak approximation.  
\end{conj}
The main technical challenge is the singular fibers.   In Section~\ref{sect:cases},
we shall survey situations where the fibers can be successfully analyzed,
or where the global geometry ensures weak approximation.

\paragraph{Converse theorems}
We have seen that rationally connected varieties often satisfy weak approximation;  in many
cases, the converse also holds.  The results in this section originate from conversations
with Jason Starr.

We start with a purely geometric result:
\begin{theo} \label{theo:break}
Let $X$ be a smooth projective variety over $k(B)$, with regular model
$\pi:\cX \ra B$.  Suppose that $\cX$ admits a section
$s:B\ra \cX$ with the following property:  Given general points $b,b' \in B$ and
general points $x\in \cX_b$ and $x' \in \cX'_b$, there exists a deformation
$s'$ of $s$ such that $s'(b)=x$ and $s'(b')=x'$.  
Then $X$ is rationally connected.  
\end{theo}
The key hypothesis can be expressed nicely in terms of stable maps (see the Appendix):  Let $\beta=[s(B)] \in H_2(\cX_{\bC},\bZ)$
and $\ocM'_{g(B),2}(\cX,\beta)$ the irreducible
component of the stable map space containing $\{s:(B,b_1,b_2) \ra \cX \}$, where 
$b_1,b_2 \in B$ are general marked
points.  We assume the evaluation mapping
\begin{equation} \label{eq:evaltwo}
\begin{array}{rcl}
\ev:\ocM'_{g(b),2)}(\cX,\beta) &\ra& \cX \times \cX\\
\{f:(C,c_1,c_2)\ra \cX \} & \mapsto & (f(c_1),f(c_2))
\end{array} 
\end{equation}
is dominant.
\begin{proof}
Suppose we have a section $s':B\ra \cX$
taking general values $x\in \cX_b$ and $x'\in \cX_{b'}$ at distinct general
points $b,b' \in B$.  This yields a two-pointed stable mapping
$$\{s':(B,b,b')\ra \cX\}$$
with $\ev(s')=(x,x')$.  

Specialize $b'\rightsquigarrow b$
and $x'\rightsquigarrow x_2\in \cX_b$ for $x_2 \neq x$. 
This induces a specialization of our stable map
$$\{s':(B,b,b')\ra \cX \} \rightsquigarrow \{f'':(B'',p,p_2) \ra \cX \},$$
where $f''(p)=x$ and $f''(p_2)=x_2$.  
Now there is a unique irreducible component $B \subset B''$ such that
the restriction $s''=f''|B$ is a section.  Let $C$ denote the union of 
irreducible components mapped into $\cX_b$.  We know:
\begin{itemize}
\item
$C$ is a tree of rational curves;
\item
$x,x_2 \in f''(C)$.
\end{itemize}
Thus we have a chain of rational curves in $\cX_b$ joining $x$ and $x_2$,
i.e., $\cX_b$ is rationally chain connected.  
\begin{figure}
\begin{center}
\includegraphics{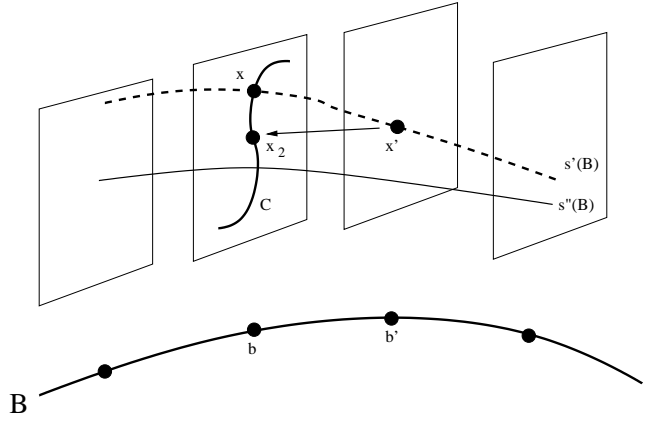}
\end{center}
\caption{Specializing sections until they break}
\end{figure}
\end{proof}

It is not hard to formulate infinitesimal criteria for when our hypothesis holds:
\begin{prop}
Let $X$ be a smooth projective variety over $k(B)$, with regular model
$\pi:\cX \ra B$ admitting a section
$s:B\ra \cX$. 
Suppose the normal bundle $N_s$ has no higher cohomology and, for general
$(b,b') \in B\times B$,
the natural differential
$$d_{b,b'}:\Gamma(B,N_s) \ra N_s|b \oplus N_s|b' \simeq T_{s(b)}\cX_b \oplus T_{s(b')}\cX_{b'}$$
is surjective.  
Then (\ref{eq:evaltwo}) is dominant.
\end{prop}
\begin{proof}
Let $\oSect(\cX/B)$ denote the irreducible component of the Hilbert scheme (or
algebraic space, if $\cX\ra B$ happens not to be a projective scheme \cite{Art69}) 
containing $s(B) \subset \cX$.  There is an open subset 
$$\Sect(\cX/B) \subset \oSect(\cX/B)$$
containing the {\em bona fide} sections.  This is unobstructed at $[s(B)]$ by our vanishing
assumption, hence
smooth of the expected dimension.  Given $b,b'\in B$ distinct, there 
is an evaluation morphism
$$\begin{array}{rcl}\Sect(\cX/B) & \ra & \cX_b \times \cX_{b'} \\
                     s' & \mapsto & (s'(b),s'(b')).
\end{array}
$$
with differential $d_{b,b'}$ (e.g., \cite[I.2]{Kolbook}). 
Since this is surjective, the evaluation
morphism is dominant.  
\end{proof}

\begin{coro}
Assume $k$ is uncountable.  Let $X$ be a smooth projective variety over $k(B)$
with regular model $\pi:\cX \ra B$.  Suppose there exist distinct places of good 
reduction $b,b' \in B$ such that, for general $x\in \cX_b$ and $x'\in \cX_{b'}$,
there exists a section $s:B \ra \cX$ with $s(b)=x$ and $s'(b)=x'$.  
Then $X$ is rationally connected.
\end{coro}
See \cite[IV.3.6]{Kolbook} for a discussion of how the assumption that $k$ is uncountable
allows us to pass from the set-theoretic condition (there exists a section through
all points of the fibers) to an algebro-geometric condition (the evaluation morphism
is dominant).

\section{Special cases of weak approximation}
\label{sect:cases}
We continue to assume $k$ is algebraically closed of characteristic zero and $B$ is a smooth
projective curve over $k$, with function field $F=k(B)$.  

\paragraph{Fibration theorems}
The following theorem of Colliot-Th\'el\`ene and Gille \cite[2.2]{CTG} is fundamental for inductive arguments:
\begin{theo} \label{theo:CTG}
Let $X^{\circ}, Y^{\circ}$ denote smooth varieties over $F$ and $f:X^{\circ} \ra Y^{\circ}$
a smooth morphism over $F$ with connected fibers.  Assume that
\begin{itemize}
\item{$Y^{\circ}$ satisfies weak approximation;}
\item{for each $y\in Y^{\circ}(F)$, $X^{\circ}_y=f^{-1}(y)$ admits a rational point 
and satisfies weak approximation.}
\end{itemize}
Then $X^{\circ}$ satisfies weak approximation.
\end{theo}
Note that the varieties need not be projective;  indeed, we shall shrink them 
so as to satisfy the smoothness hypothesis.  
The idea of the proof is quite natural:  Approximate first in the base, then in the
fibers.  

Here is a quick but very useful consequence:
\begin{coro} \label{coro:conic}
Let $X$ be a smooth projective variety over $F$, $Y$ a smooth projective variety rational over $F$,
and $f:X\ra Y$ a conic fibration, i.e., a morphism whose geometric generic fiber is isomorphic to $\bP^1$.  
Then $X$ satisfies weak approximation.
\end{coro}
Here, take $Y^{\circ} \subset Y$ to be the open subset over which $f$ is smooth and
$X^{\circ}=f^{-1}(Y^{\circ})$.  
The base satisfies weak approximation by Corollary~\ref{coro:rational}.
The fibers are all isomorphic to $\bP^1$---indeed, conics over $F$ are automatically split---hence
also satisfy weak approximation.

\paragraph{Classification of surfaces and weak approximation}
The birational classification of rational surfaces over a non-closed field is due to Enriques,  Manin \cite{Manin}, and Iskovskikh \cite{Isk}.
Let $X$ be a smooth projective surface over $F$ that is geometrically rational.  Assume that $X$ is {\em minimal}, in the 
sense that it admits no birational morphisms $\phi:X \ra X'$ to a nonsingular projective surface, defined over $F$.
Then $X$ is isomorphic to one of the following:
\begin{itemize}
\item{$\bP^2$ or a quadric $Q\subset \bP^3$;}
\item{a Del Pezzo surface with $\Pic(X)=\bZ K_X$, of degree $d=K_X \cdot K_X$;}
\item{a conic bundle over $\bP^1$ with $\Pic(X)\simeq \bZ \oplus \bZ$.}
\end{itemize}
When $X=\bP^2$, a quadric, or a Del Pezzo surface of degree $d\ge 5$, then $X$ is rational 
and satisfies weak approximation by Corollary~\ref{coro:rational} (cf. \cite[\S 3.5]{HaCMI}).  
The conic bundle cases are covered by Corollary~\ref{coro:conic}.  The case of degree four Del Pezzo surfaces falls
under this category (see \cite[\S 2]{CTG}):  If $X$ is a degree four Del Pezzo surface over $F=k(B)$ then $X(F)$
is Zariski dense by the Koll\'ar-Miyaoka-Mori theorem (Theorem~\ref{theo:KMM}).  For suitable $x\in X(F)$, the blow-up $X'=\mathrm{Bl}_x(X)$ is a 
cubic surface containing a line $\ell$, i.e., the exceptional divisor over $x$.  Projecting from $\ell$ gives a conic
bundle structure
$$\pi_{\ell}:X' \ra \bP^1;$$
weak approximation for $X'$ and $X$ therefore follows.

Thus to complete the case of rational surfaces, it remains to prove weak approximation for Del Pezzo surfaces of degree $3$, $2$, and $1$.  

\begin{rema}[suggested by Colliot-Th\'el\`ene]
Some of the discussion of degree four del Pezzo surfaces 
above extends to more general fields.
If such a surface admits a rational point not lying on a line then
it is unirational \cite[Thm.~29.4]{Maninbook}.  And over an infinite perfect
field, a degree four del Pezzo surface with a rational point admits
a rational point not lying on a line (cf. \cite[Thm.~30.1]{Maninbook}).  
In particular, we can find a rational point such that projection from
that point yields a smooth cubic surface.
\end{rema}

\paragraph{Hypersurfaces of very low degree}
Here is another example of how Theorem~\ref{theo:CTG} can be profitably applied:
\begin{theo} \label{theo:LD} \cite{HT09}
Let $X\subset \bP^n$ be a smooth hypersurface of degree $d$.  Define 
$\phi:\bN \ra \bN$ by the recursive formula
$$\phi(1)=1, \quad \phi(d)=\binom{\phi(d-1)+d-1}{\phi(d-1)}, d>1.$$
Then $X$ satisfies weak approximation if $n \ge \phi(d)$.  
\end{theo}
Tabulating 
$$\begin{array}{r|cccc} 
d & 1 & 2 & 3 & 4 \\
\hline
\phi(d) & 1 & 2 & 6 & 84 
\end{array}
$$
we see that $\phi(d) \gg d^2$ for $d \ge 4$.  Theorems~\ref{theo:WARSC1} and \ref{theo:WARSC2} below
give stronger results for $d\ge 4$.   
We therefore focus on Theorem~\ref{theo:LD} in the special case $d=3$.  

Here is the idea:  Let $F_1(X) \subset \bG(1,n)$ denote the variety of lines on $X$, which is smooth of
the expected dimension \cite[1.12]{AK}.  An application of the adjunction formula shows 
$$\omega_{F_1(X)} \simeq \cO_X(5-n),$$
so $F_1(X)$ has ample anticanonical class  (`the Fano variety of lines is Fano').
It follows \cite[Thm.~0.1]{KMMb} that $F_1(X)$ is rationally connected and thus has an $F$-rational point $\Lambda$
by the Graber-Harris-Starr Theorem (Theorem~\ref{theo:GHS}).  

Consider the projection from $\Lambda$
$$\pi_{\Lambda}: \bP^n \dashrightarrow \bP^{n-2}$$
and the induced
$$\pi_{\Lambda}:\Bl_{\Lambda}X \ra \bP^{n-2},$$
which is a conic bundle.  Corollary~\ref{coro:conic} implies $X$ satisfies weak approximation.


\paragraph{Cubic surfaces with mild singular fibers}
We describe situations where Theorem~\ref{theo:WASRC} applies to cubic surfaces:

\begin{prop} \label{prop:cubicRDP} \cite[\S 5]{HT08}
Let $Y$ be a cubic surface with only rational double points.  Then the smooth
locus $Y^{sm} \subset Y$ is strongly rationally connected.   
\end{prop}
We sketch this in the special case where $Y$ has
a single ordinary double point $y_0$.  
\begin{proof}
We apply Proposition~\ref{prop:thrupts}:  Any rational
curve in $Y^{sm}$ that meets the maximal strongly rationally connected subset of
$Y^{sm}$ is contained in that subset.  Thus
given $y \in Y^{sm}$, it suffices to exhibit a rational curve joining
$y$ to a general point $y' \in Y^{sm}$.   The fact that $y'$ is general implies
\begin{itemize}
\item
If $\ell$ is the line
containing $y$ and $y'$, then $\ell$ meets $Y$ transversally at
three distinct points $\{y,y',y''\}$.
\item
There are no lines on $Y$ through $y'$ or $y''$.
\item
There exists no hyperplane section containing $y'$ or $y''$, 
but not containing $y_0$, and consisting of a line and 
a conic meeting tangentially.
\end{itemize}

Consider the pencil of hyperplane
sections containing $\ell$, which induces an elliptic fibration
$$\varphi:\Bl_{y_0,y,y',y''}Y \ra \bP^1.$$
It suffices to exhibit an irreducible singular fiber of $\varphi$.
Now $\varphi$ has a total of twelve singular fibers, counted with multiplicities.

First, consider the fiber $F_0$ corresponding to the hyperplane section $H_0$ through $y_0$.
The possibilities for $H_0$ are:
\begin{enumerate}
\item{an irreducible plane cubic with a node at $y_0$;}
\item{an irreducible plane cubic with a cusp at $y_0$;}
\item{the union of a line and a conic meeting transversally at $y_0$;}
\item{the union of a line and a conic meeting tangentially at $y_0$.}
\end{enumerate}
Note that $F_0$ consists of the exceptional curve over $y_0$ and the proper
transform of $H_0$.  The corresponding possibilities are:
\begin{enumerate}
\item{the union of two $\bP^1$'s meeting transversally in two points;}
\item{the union of two $\bP^1$'s meeting tangentially;}
\item{the union of three $\bP^1$'s meeting pairwise transversally;}
\item{the union of three $\bP^1$'s coincident at a point.}
\end{enumerate}
The multiplicity of $F_0$ is $2,3,3$, and $4$ respectively;  these are computed by
summing the Milnor numbers of the corresponding singularities \cite[2.8.3]{Te}.

There might be up to three additional
reducible fibers beyond $F_0$;  these are unions of lines through $y$ and conics through
$y'$ and $y''$, meeting transversally.  Altogether, these contribute at most six to the 
multiplicity count.  
In order to account for all twelve singular fibers, there must be at least one that
is irreducible and disjoint from $y_0$.  
\end{proof}

Here is the application to weak approximation:
\begin{prop}
Let $X$ be a cubic surface over $F=k(B)$ and $\pi:\cX \ra B$ a regular model.  
Suppose that $S \subset B$ is a finite set with complement $B^{\circ}$, chosen such that
for $b\in B^{\circ}$ of bad reduction,
$\cX_b$ is a cubic surface with 
rational double points.  Then $X$ satisfies weak approximation at places of $B^{\circ}$.  
\end{prop}
Here we apply Theorem~\ref{theo:WASRC} to the open subset $\cY \subset \cX\times_B B^{\circ}$ where
$\pi$ is smooth.  Proposition~\ref{prop:cubicRDP} gives strong rational connectedness;
Theorem~\ref{theo:GHS} gives the required section.

\paragraph{Further applications to mild singular fibers}
We list further cases where this line of reasoning applies.  Let $Y$ be a projective variety
with smooth locus $Y^{sm}$.  Then $Y^{sm}$ is strongly rationally connected provided
\begin{enumerate}
\item{$Y$ is a degree two Del Pezzo surface with certain types of rational double
points; this is due to Knecht \cite{Kn1,Kn2}.}
\item{$Y$ is a log Del Pezzo surface, including surfaces with quotient singularities and ample anticanonical
class;  this is a result of C. Xu \cite{Xu}, and is applied to weak approximation
questions in \cite{Xu2}.}
\item{$Y\subset \bP^n$ is a hypersurface of degree $d\le n$ with isolated terminal singularities \cite[\S 6]{HT08};
this includes hypersurfaces of dimension $\ge 3$ with ordinary double points, i.e., with local equation
$$x_1^2+x_2^2+\cdots + x_n^2 + \text{ higher order terms }=0.$$}
\end{enumerate}
Combining the last example with the discussion of cubic surfaces, we obtain:
\begin{coro}[Weak approximation for general Fano hypersurfaces]  \cite{HT08}
Let $X \subset \bP^n$ be a hypersurface of degree $d\le n$ over $F=k(B)$, with square-free discriminant.
Then $X$ satisfies weak approximation.
\end{coro}
\begin{proof}  
The condition that the discriminant is square-free implies $X$ admits a regular projective model
such that the singular fibers each have a single
ordinary double point.  Indeed, the multiplicity of the discriminant at a point equals the sum of the Milnor numbers 
of the singularities in the corresponding fiber \cite{Te}; an isolated singularity has Milnor number one precisely when it is an
ordinary double point.  
\end{proof}

\paragraph{Non-regular models}
Let $X$ be a smooth projective rationally connected variety over $F=k(B)$
and $\pi:\cX \ra B$ a model.  When $\cX$ is regular, we saw in Section~\ref{sect:elements}
that sections of $\pi$ necessarily factor through $\cX^{sm} \subset \cX$, the locus
where $\pi$ is smooth.  However, when $\cX$ is not regular it may admit sections passing
through singularities over places of bad reduction.  

There are a number of approaches we can take to deal with this eventuality.  The 
most direct is to resolve singularities
$\varrho: \widetilde{\cX} \ra \cX$ to obtain a regular model
$\widetilde{\pi}:\widetilde{\cX} \ra B$.  
Sections $\sigma:B \ra \cX$ through singularities lift to sections 
$\widetilde{\sigma}:B \ra \widetilde{\cX}$ meeting exceptional divisors of
$\varrho$ lying over those singularities.  When this resolution can be 
performed explicitly, we can attempt to construct sections meeting prescribed
exceptional divisors.  However, this approach requires precise control over the
rational curves in various homology classes of the fibers of $\widetilde{\pi}$.

Here is a result that can be proven in this way:
\begin{theo} \cite[Thm. 18]{HT09} \label{theo:cA1}
Let $\cX \ra B$ be a model of a smooth projective rationally connected variety $X$.  
Assume that for each place $b\in B$ of bad reduction, the singular fiber 
$\cX_b$ has the following properties:
\begin{itemize}
\item{$\cX_b$ has only ordinary double points;}
\item{$\cX_b^{sm}$ is strongly rationally connected;}
\item{if $\rho:\widetilde{\cX_b} \ra \cX_b$ is the blow up of the double points, then 
for each component $D$ of the exceptional locus $\mathrm{Exc}(\rho)$ there exists a rational curve in 
$\widetilde{\cX_b}$ meeting $D$ at one point transversally  and avoiding $\mathrm{Exc}(\rho)\setminus D$.}
\end{itemize}
Then $\cX \ra B$ satisfies weak approximation.
\end{theo}

Unfortunately, this approach is not robust enough to fully prove weak approximation, even
for cubic surfaces.
\begin{exam}
Consider the Cayley cubic surface
$$Y=\{wxy+xyz+yzw+zwx=0 \} \subset \bP^3,$$
which has ordinary double points at $[1,0,0,0]$,$[0,1,0,0]$,$[0,0,1,0]$, and $[0,0,0,1]$.  Let
$\rho:\widetilde{Y} \ra Y$ denote the minimal resolution with
$$\mathrm{Exc}(\rho)=D_1 \sqcup D_2 \sqcup D_3 \sqcup D_4, \quad D_j \simeq \bP^1, D_j\cdot D_j =-2.$$
We may interpret $\widetilde{Y}$ as a blow-up of $\bP^2$:  Choose four lines $L_1,L_2,L_3,L_4 \subset \bP^2$
in general position, with intersections $p_{ij}=L_i \cap L_j$.  Then 
$\widetilde{Y}\simeq \Bl_{p_{12},\ldots,p_{34}}(\bP^2)$ with exceptional curves $E_{12},\ldots,E_{34}$.  
Furthermore, the proper transform of $L_j$ equals $D_j$.  If $L$ is the pullback of the hyperplane class
of $\bP^2$ to $\widetilde{Y}$ then $D_j=L-E_{ja}-E_{jb}-E_{jc}$ where $\{j,a,b,c\}=\{1,2,3,4\}.$
Thus we have
$$[\mathrm{Exc}(\rho)]=[D_1+D_2+D_3+D_4]=4L-2(E_{12}+\cdots +E_{34}),$$
which is divisible by two.  This is incompatible with the last assumption of Theorem~\ref{theo:cA1}.
\end{exam}

\section{Weak approximation and rationally simply connected varieties}
\label{sect:RSC}
There are large classes of varieties where weak approximation can be proven at all places.  
Here we focus on the 
`rationally simply connected varieties' introduced by Barry Mazur and
studied systematically by de Jong and Starr \cite{dJS}.  
\paragraph{An easier argument under strong assumptions}
\begin{theo} \label{theo:WARSC1}
Let $X$ be a smooth projective variety over $F=k(B)$.  Assume that for each $m\ge 1$,
there exists a class $\beta \in H_2(X_{\bC},\bZ)$ and 
an irreducible component of the space of Kontsevich stable maps
$$M \subset \ocM_{0,m}(X,\beta),$$
with the following properties:
\begin{itemize}
\item{$M$ is defined and absolutely irreducible over $F$;}
\item{a general point of $M$ parametrizes a smooth immersed curve;}
\item{the evaluation morphism
$$\begin{array}{rcl}
\ev: M & \ra & X^m \\
\{f:(C,p_1,\ldots,p_m)\ra X \} & \mapsto & (f(p_1),\ldots,f(p_m)) 
\end{array}
$$
is dominant with rationally connected generic fiber.}
\end{itemize}
Then $X$ satisfies weak approximation.
\end{theo}
The key hypothesis is known to hold for
\begin{enumerate}
\item{$X \subset \bP^n$ a complete intersection of degrees $(d_1,\ldots,d_r)$
with $d_1\ge d_2\ge \ldots \ge d_r \ge 2$, provided
$n+1 \ge \sum_{i=1}^r (2d_i^2-d_i)$ and $\deg(\beta) \ge 4m-6$ \cite[1.2]{dJS}.}
\item{$X\subset \bP^n$ is a general hypersurface of degree $d$, provided
$n \ge d^2$ and $\deg(\beta) \gg 0$ \cite[1.3]{dJS}\cite{Starr06}.}
\end{enumerate}
These classes of 
varieties are said to be {\em strongly rationally simply connected}.  
\begin{coro}
Let $X\subset \bP^n$ be a smooth complete intersection of type $(d_1,\ldots,d_r)$ over $F$.
If 
$n+1 \ge \sum_{i=1}^r (2d_i^2-d_i)$ then $X$ satisfies weak approximation.  
\end{coro}
\begin{proof}
We apply the iterated blow-up construction, as formulated in
Observation~\ref{obse2}. 
Suppose $\cX \ra B$ is a regular projective model of $X$,
$b_1,\ldots,b_r \in B$ distinct points, and $x_i \in \cX_{b_i}^{sm}, i=1,\ldots,r$.
We produce a section $\sigma:B\ra \cX$ with 
$\sigma(b_i)=x_i$ for $i=1,\ldots,r$.  

There exists a multisection 
$$\begin{array}{rcl}
\cZ & \subset & \cX \\
    & \searrow & \downarrow \\
    &	 	& B 
\end{array}
$$
such that 
$x_1,\ldots,x_r \in \cZ$ and $\cZ \ra B$ is
unramified at those points.  Indeed, embed $\cX \subset \bP^N$
and take a general one-dimensional complete intersection containing
$x_1,\ldots,x_r$.  Set $m=\deg(\cZ/B)$ and let $Z\subset X$
denote the corresponding zero-cycle of length $m$.  

Suppose, for the moment, that $Y:=\ev^{-1}(Z)$ is rationally connected with generic point
corresponding to a stable map with the following properties:
\begin{itemize}
\item{the domain is $\bP^1$;}
\item{the mapping is an immersion;}
\item{at the marked points the mapping is an embedding, i.e., at these points
it is an isomorphism onto its image.}
\end{itemize}
Let $\cY \ra B$ be a model of $Y$.
The Graber-Harris-Starr and Koll\'ar-Miyaoka-Mori theorems
give a Zariski-dense collection of sections $\rho:B \ra \cY$.  
Given any subvariety $\Delta \subsetneq Y$, we may assume $\rho(B)$ avoids $\Delta$
in the generic fiber.
In particular, $\rho$
corresponds to an immersed rational curve 
$$Z\subset R \subset X$$
with the properties prescribed above.

Let $\cR\subset \cX$ denote the closure of $R$, a ruled surface
containing $\cZ$
and thus $x_1,\ldots,x_r$.  Let $\psi:\widetilde{\cX} \ra \cX$ be an embedded resolution
of singularities for $\cR$, with proper transform $\widetilde{\cR} \subset \cX$.  
Since $R$ is smooth along $Z$, $\psi$ is an isomorphism over $Z$;
let $\widetilde{\cZ} \subset \widetilde{\cR}$
denote the closure of $Z$.  Now $\widetilde{\cZ} \ra \cZ$ is
an isomorphism over the points $\{x_1,\ldots,x_r\}$;
let $x'_1,\ldots,x'_r\in \widetilde{\cZ}$ denote their pre-images.  
Furthermore, $\widetilde{\cR} \ra B$ is smooth at these points, because
there is a formal section through each $x'_j$, e.g., $\widetilde{\cZ}$.
(Recall that $\widetilde{\cZ}\ra B$ is unramified at $x'_j$.)  
Weak approximation holds for rational curves, 
thus there is a section $\tilde{\sigma}:B\ra \widetilde{\cR}$ with $\tilde{\sigma}(b_i)=x'_i$;
$\sigma=\psi\circ \tilde{\sigma}$ is our desired section.  

\begin{figure}
\begin{center}
\includegraphics{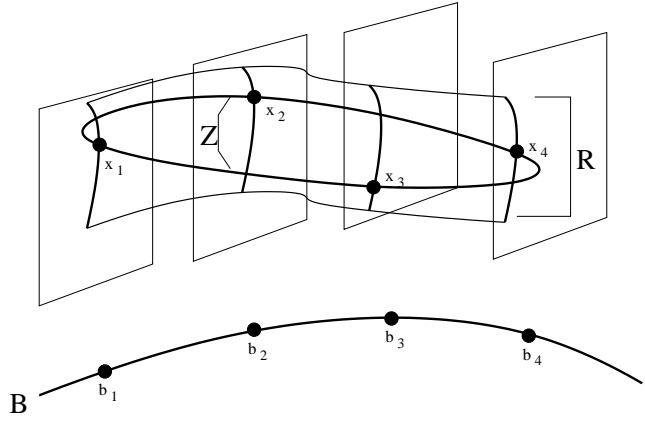}
\end{center}
\caption{Multisection/ruled surface construction}
\end{figure}

We now illustrate how to ensure that $Y=\ev^{-1}(Z)$ is rationally connected with generic point
corresponding to a smooth rational with an immersion that is an
embedding at the marked points.  There exists a closed subset $\Delta' \subsetneq X^m$
where one or more of these conditions fails.  We need to choose $Z$ so as to avoid this subset.  
However, after attaching a suitably large number of fibral very free curves
to $Z$ (cf. Proposition~\ref{prop:nicesection} and \cite[\S 2]{GHS}),
we can find a degree $m$ cycle $Z'\subset X$ 
avoiding $\Delta'$ but still containing $x_1,\ldots,x_m$.  Indeed, let $b\in B$ denote a general
point with $z_1,\ldots,z_m \in \cZ_b$.  Then there points $z'_1,\ldots,z'_q \in \cZ$ and free
curves $f_i:\bP^1 \ra \cX_{\pi(z'_i)}$ such that the comb
$$C=Z \cup_{z'_1} f_1(\bP^1) \cup \cdots \cup_{z'_q} f_q(\bP^1)$$
has $N_{C/\cX}\otimes \cI_{z_1,\ldots,z_m,x_1,\ldots,x_r}$ globally generated
with no higher cohomology.  The general
deformations of $C$ through $x_1,\ldots,x_r$ is a multisection meeting $X$ in a general cycle of length $m$. 
\end{proof}

\paragraph{A harder argument under weaker assumptions}
We are grateful to Jason Starr for suggesting this argument.  

Before we state our result, we review some facts on
mappings from rational curves to smooth varieties, following
\cite[\S 3]{dJS}.  

We shall use the following basic fact repeatedly:  Let $W$ be an
integral scheme of finite type over a field $F$ of characteristic zero.
Suppose there exists a geometrically connected nonempty subscheme $V\subset W$
defined over $F$, along which $W$ is normal.  Then $W$ is geometrically integral.
Indeed, all the geometric irreducible components of $W$ meet along $V$;
$W$ would fail to be normal along $V$ if there were more than one.

Let $Y$ be a smooth proper variety over an algebraically closed
field of characteristic zero. 
\begin{prop}
Fix a free curve in $Y$ and let
$$\ocM'_{0,0}(Y,\beta) \subset \ocM_{0,0}(Y,\beta)$$
denote the unique irreducible component of the stable map space 
containing this curve.  Then for each $m\ge 1$ there is a distinguished irreducible
component
$$\ocM'_{0,0}(Y,m\beta) \subset \ocM_{0,0}(Y,m\beta)$$
characterized by the following property:
For each finite morphism $g:\bP^1 \ra \bP^1$ of degree $m>0$ and 
free $\{f:\bP^1 \ra Y \} \in \ocM'_{0,0}(Y,\beta)$, the composition
$$\{(f\circ g):\bP^1 \ra Y \} \in \ocM'_{0,0}(Y,m\beta).$$
\end{prop}
We write $\ocM'_{0,n}(Y,m\beta)$ for the corresponding irreducible
component of the pointed moduli space.  
\begin{proof}
First, the uniqueness of $\ocM'_{0,0}(Y,\beta)$ is a consequence of the
fact that the stable map space is smooth at a free rational curve, because
the normal bundle is a quotient of the pull-back of $T_Y$.
Fix $m \ge 1$.  
The space of all degree $m$ morphisms $\bP^1 \ra \bP^1$ is irreducible,
so there is an irreducible scheme $R_m$ parametrizing the composed maps $f \circ g$
described above.  Each
$f \circ g:\bP^1 \ra Y$ remains free;
indeed, if $f^*T_Y \simeq \oplus_{j=1}^r \cO_{\bP^1}(a_j)$ then
$(f\circ g)^*T_Y \simeq \oplus_{j=1}^r \cO_{\bP^1}(ma_j)$.
Hence $\ocM_{0,0}(Y,m\beta)$ is smooth along $R_m$, and thus admits a distinguished
irreducible component containing $R_m$.  
\end{proof}
\begin{rema}
Suppose $f:\bP^1 \ra Y$ is as described above, but defined over an arbitrary
field $F$ of characteristic zero.  Then $\ocM'_{0,0}(Y,m\beta)$ is defined over $F$ as well.
\end{rema}

\begin{defi}
Let $C$ be a nodal connected curve of genus zero and $Y$ a smooth proper variety.
A morphism $h:C \ra Y$ is {\em admissible} if, for each irreducible $C_i \subset C$,
the restriction $h|C_i$ is either constant or a free immersion.  
\end{defi}
Each non-constant admissible morphism admits a deformation to a free curve
\cite[II.7.6]{Kolbook};
the moduli stack of stable maps is smooth 
at such points
(see the argument of \cite[5.2]{FP}).

\begin{prop}  \label{prop:prep}
Let $\ocM'_{0,0}(Y,\beta)$ denote an irreducible component of the
stable map space containing a free curve, 
such that the evaluation map
$$\ev:\ocM'_{0,1}(Y,\beta) \ra Y$$
is dominant with irreducible geometric generic fiber.  
Let $h:C \ra Y$ be an admissible stable map such that, for each irreducible
component $C_i \subset C$ with $h|C_i$ non-constant, the restriction
$h_i:=h|C_i:\bP^1 \ra Y$ is in $\ocM'_{0,0}(Y,\beta)$.
Then
\begin{itemize}
\item{$h \in \ocM'_{0,0}(Y,m\beta)$, where $m$ is the number of non-contracted components;}
\item{the evaluation map
$$\ev_m:\ocM'_{0,1}(Y,m\beta) \ra Y$$
also is dominant with irreducible geometric generic fiber.}
\end{itemize}
\end{prop}
Thus $\ocM'_{0,0}(Y,m\beta)$ contains smoothings of chains and trees
of $m$ free curves
from $\ocM'_{0,0}(Y,\beta)$

Proposition~\ref{prop:prep} a special case of \cite[3.5]{dJS}, so we only sketch
the proof.
\begin{proof}
Our proof is by induction on $m$.  There is nothing to prove in the base case
$m=1$, so we focus on the inductive step.

\begin{figure}
\begin{center}
\includegraphics{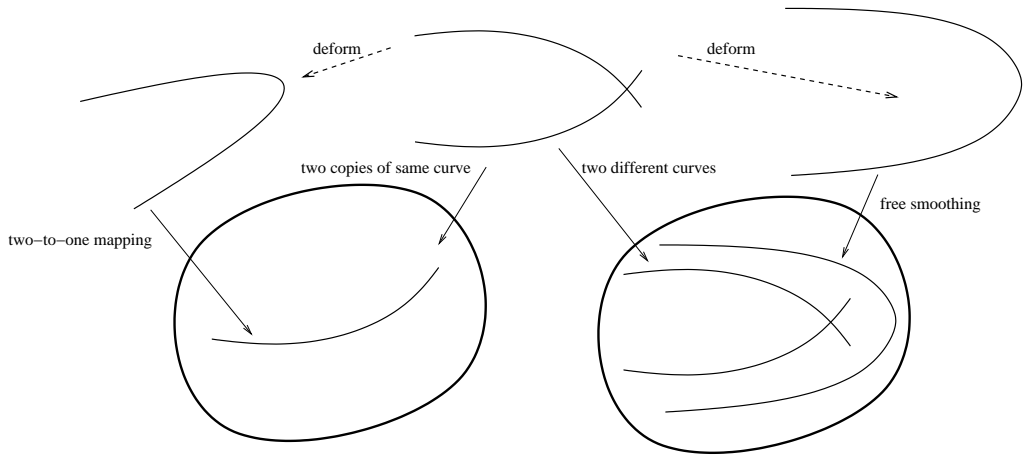}
\end{center}
\caption{The key degenerations}
\end{figure}

Fix a dense open subset $U\subset Y$ over
which the fibers of 
$$\ev_{n}|M'_{0,1}(Y,n\beta) \ra Y, \ n=1,\ldots,m-1,$$
are all irreducible of the expected dimension
and contain free curves.  
It follows that for any $y \in U$ and smooth points
$\{g_1:(D_1,p_1) \ra Y \} \in \ev_{n}^{-1}(y)$ and
$\{g_2:(D_2,p_2) \ra Y \} \in \ev_{m-n}^{-1}(y)$,  the stable map
$$\{g=(g_1,g_2):D_1 \cup_{p_1=p_2} D_2 \ra Y \} \in \ocM_{0,0}(Y,m\beta)$$
is in the same irreducible component of the stable map
space.  

This distinguished component must coincide with
$\ocM'_{0,0}(Y,m\beta)$.  Indeed, it suffices to show
that the $g$ constructed in the previous paragraph can be chosen
to lie on $\ocM'_{0,0}(Y,m\beta)$.  We may take $g_1$ and $g_2$
to be suitably general branched coverings of a fixed free curve 
$\{f:\bP^1 \ra Y \} \in M'_{0,0}(Y,\beta)$ through $y$.  However, the moduli space
of stable maps of degree $m$ from $\bP^1$ to $\bP^1$ is irreducible,
and even smooth as a stack.  Thus $g$ deforms to a degree $m$ covering
composed with $f$.  

Consider our admissible mapping $h:C \ra Y$.  We may assume $h$ maps the nodes
of $C$ to $U$.  Pick $D_1$ to be an `extremal' irreducible
component of $C$, i.e., one corresponding to a $1$-valent vertex of the dual
graph to $C$.  The stability condition means this is not contracted.
Set $D_2=\overline{C\setminus D_2}$ and $p=D_1 \cap D_2$ the disconnecting
node joining them.  Taking $g_1=h|D_1$ and $g_2=h|D_2$, we see that the
$g$ constructed above coincides with $h$, and thus is in the desired component.  

It remains to show that the generic fiber $\ev_m^{-1}(y)$ is irreducible.  
However, it admits a distinguished stratum
$$\ev^{-1}(y) \times \cM_{0,m+1},$$
corresponding to attaching $m$ copies of a fixed mapping
$$\{r:(\bP^1,p) \ra Y \} \in \ev^{-1}(y)$$
to a single contracted component $(\bP^1,q,p_1,\ldots,p_m)$.  
Here the $j$th copy is attached to $p_j$.  
This stratum is geometrically irreducible and $\ev_m^{-1}(y)$ is smooth 
at its generic point, thus $\ev_m^{-1}(y)$ is irreducible.  
\end{proof}

\begin{theo} \label{theo:WARSC2}
Let $X$ be a smooth projective variety over $F=k(B)$.  Assume 
there exists a class $\beta \in H_2(X_{\bC},\bZ)$ and 
an irreducible component of the space of Kontsevich stable maps
$$M:=\ocM'_{0,2}(X,\beta) \subset \ocM_{0,2}(X,\beta)$$
with the following properties:
\begin{itemize}
\item{$M$ is defined and absolutely irreducible over $F$;}
\item{a general point of $M$ parametrizes an immersed smooth curve;}
\item{the evaluation morphism
$$\begin{array}{rcl}
\ev: M & \ra & X^2 \\
\{f:(C,p_1,p_2)\ra X) \} & \mapsto & (f(p_1),f(p_2))
\end{array}
$$
is dominant with rationally connected generic fiber.}
\end{itemize}
Assume furthermore that the resulting components
$M_m:=\ocM'_{0,0}(X,m\beta)$ 
are rationally connected for each $m\ge 1$.  Then 
$X$ satisfies weak approximation.
\end{theo}
The main hypothesis holds for 
$X \subset \bP^n$ a complete intersection of degrees $(d_1,\ldots,d_r)$
with $d_1\ge d_2\ge \ldots \ge d_r \ge 2$, provided
$n+1 \ge \sum_{i=1}^r d_i^2$, $n\ge 4$, and $\deg(\beta)\ge 2$ \cite[1.1]{dJS}.
These properties are established in the course of proving these
varieties are {\em rationally simply connected}, i.e., for each $m \ge 2$ there is a 
canonically-defined irreducible component
$$\ocM'_{0,2}(X,m) \subset \ocM_{0,2}(X,m)$$
of the space of degree $m$ two-pointed stable maps of genus zero, such that
$$\ev:\ocM'_{0,2}(X,m) \ra X \times X$$
is dominant with rationally connected generic fiber.  
\begin{coro}
Let $X\subset \bP^n$ be a smooth 
complete intersections of type $(d_1,\ldots,d_r)$ over $F$.
If $n+1 \ge \sum_{i=1}^r d_i^2$ then $X$ satisfies weak approximation.
\end{coro}

\begin{proof}
The first steps of the argument are identical to Theorem~\ref{theo:WARSC1};  thus we 
retain the notation of its proof, i.e., $\cZ \ra B$ is the degree $m$ multisection
containing the points we seek to approximate, but otherwise general.  
Let $U\subset \cX \times_B \cX$ denote the open subset over which $\ev$
has (geometrically) irreducible rationally connected fiber containing a very free 
immersed rational curve.  
We may assume $\cZ$ meets the image of $U$ under the first projection.

Fix a section $s:B \ra \cX$ meeting the image of $U$ under the section
projection.  Consider the basechange
$$\cX \times_B \cZ \ra \cZ$$
which has {\em two} distinguished sections:  the base change $s_{\cZ}$ of $s$ and the diagonal
section
$$s':\cZ \hookrightarrow \cZ \times_B \cZ  \hookrightarrow \cX \times_B \cZ.$$
Let $F'=k(\cZ)$ denote the resulting degree $m$ extension of $F=k(B)$ and 
$M_{F'}=M\times_{\Spec(F)} \Spec(F')$.  We still have that $\ev^{-1}(\cU \times_B \cZ)$ is rationally connected.  

Let $\ocM'_{0,2}(\cX\times_B \cZ, \beta)$ denote the irreducible component of the stable map space
with generic fiber $M_{F'}$;  it is proper over $\cZ$.   
The same holds true for
$$\cP:=\ev^{-1}(s_{\cZ},s') \subset \ocM'_{0,2}(\cX \times_B \cZ,\beta),$$
where here we take the relative evaluation map.  
Furthermore, by our assumption this is a family of rationally connected 
varieties.  

The Graber-Harris-Starr and Koll\'ar-Miyaoka-Mori theorems 
imply $\cP \ra \cZ$ has a section
$\tau:\cZ \ra \cP$ and these sections are Zariski dense.  In particular, we may assume
for general $z\in \cZ$, $\tau(z)$ parametrizes a smooth very free 
rational curve $T$ joining $s_{\cZ}(z)$ and $s'(z)$.  We have
$$\begin{array}{rcccl}
(s_{\cZ}(\cZ),s'(\cZ)) & \subset &  \cT' & \subset & \cX\times_B \cZ \\
		     &         &      & \searrow & \downarrow \\
		     &	       &      &          &  \cZ,
\end{array}
$$
where $\cT'$ is the closure of $T$.  
Let $\cT$ denote the image of $\cT'$ under projection to $\cX$;  this contains both
$s$ and $\cZ$, which is the projection of $s'$.  
\begin{figure}
\begin{center}
\includegraphics{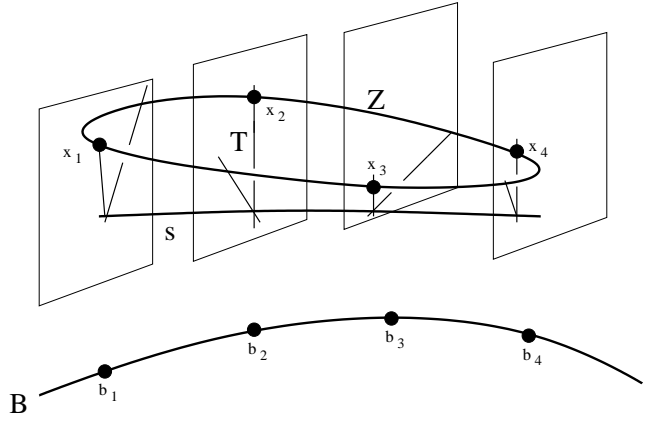}
\end{center}
\caption{Connecting the multisection to the section}
\end{figure}

As in the proof of Theorem~\ref{theo:WARSC1}, after blowing up $\cX$ we may assume $\cT \ra B$ is smooth wherever
the proper transform $\cZ \ra B$ is \'etale, and in particular, at $x_1,\ldots,x_r$.  

Let $b\in B$ be the general point;  we realize $\cT_b$ as the image of a stable map
$\{g:C \ra \cX_b \} \in \ocM_{0,0}(\cX_b, m\beta)$.  First, let
$$C=C_0 \cup C_1 \cup \cdots \cup C_m, \quad C_i \simeq \bP^1,$$
where $C_0$ meets $C_i,i=1,\ldots,m$ in a node $p_i$  and the $C_1,\ldots,C_m$
are disjoint.  Write $\cZ_b=\{z_1,\ldots,z_m\}$ and let $\cT_j$ denote the
image of $\cT'_{z_j}$ in $\cX_b$.  Since $\cT_j$ is a very free immersed curve, we
take $g|C_j:\bP^1 \ra \cT_j,j=1,\ldots,m$ to be the normalization, chosen so that $g(p_j)=s(b)$.
We take $g|C_0$ to be constant.  In the special case where $m=2$, we simply omit the
component $C_0$;  should $m=1$, we have a section through our prescribed points and there
is nothing to prove.

We apply Proposition~\ref{prop:prep}, which allows us to interpret $g$ as a non-stacky,
smooth point of $\ocM'_{0,0}(\cX_b,m\beta)$.
Let $\ocM'_{0,0}(\cX,m\beta)$ denote the corresponding irreducible component of
the space of stable maps in $\cX$.  This
is proper over $B$ with irreducible fiber $\ocM'_{0,0}(\cX_b,m\beta)$ over $b$.  
Our second hypothesis shows 
$$\ocM'_{0,0}(\cX,m\beta) \ra B$$
is a rationally connected fibration.

From now on, we tacitly replace $\ocM'_{0,0}(\cX,m\beta)$ with a resolution of singularities
that is an isomorphism over the smooth point $g$.  This resolution is also a rationally connected
fibration over $B$, and admits a section $\gamma$ with $\gamma(b)=g$.  
Applying the technique used in Proposition~\ref{prop:nicesection}, we construct a comb $C$
in $\ocM'_{0,0}(\cX,m\beta)$ with handle 
$\gamma(B)$ and teeth in suitable fibers, such that $N_C\otimes \cI_{\gamma(b_1),\cdots,\gamma(b_r)}$ is globally generated
with vanishing higher cohomology.  We deform to get a new section
$$\rho:B \ra \ocM'_{0,0}(\cX,m\beta),$$ with 
$\rho(b')$ a free immersed curve in $\ocM'_{0,0}(\cX_{b'},m\beta)$ for general $b' \in B$.  Applying this argument 
to the blow-up at $x_1,\ldots,x_r$,
we may even assume $\rho(b_i)=\gamma(b_i)$  for $i=1,\ldots,r$.  

Let $\cR \ra B$ denote the ruled surface corresponding to $\rho$.  By construction, we have $\cR_{b_i}=\cT_{b_i}$
is smooth at $x_i$ for $i=1,\ldots,r$.  Weak approximation holds in dimension one, e.g., for $\cR \ra B$, thus
we have a section $\sigma:B\ra \cR\hookrightarrow \cX $ with $\sigma(b_i)=x_i$.  
\end{proof}

\paragraph{Connections to $R$-equivalence}
Rational simple connectedness has implications for $R$-equivalence as well.  
\begin{defi} \cite{Maninbook}
Let $X$ be a projective variety over a field $F$.  Two points $x_1,x_2 \in X(F)$
are {\em directly $R$-equivalent} if there exists a morphism
$f:\bP^1 \ra X$
defined over $F$, such that $f(0)=x_1$ and $f(\infty)=x_2$.  
The resulting equivalence relation is known as {\em $R$-equivalence};  the set
of equivalence classes is denoted $X(F)/R$.  
\end{defi}

The following nice result of Pirutka \cite{Pirutka} is proven using techniques similar
to Theorems~\ref{theo:WARSC1} and \ref{theo:WARSC2}:  
\begin{theo} 
Let $F=\bC(B)$ denote the function field of a smooth complex curve $B$, or the field
$\bC((t))$.  Let $X$ be a smooth complete intersection of $r$ hypersurfaces in $\bP^n$
of degrees $d_1,\ldots,d_r$, defined over $F$.  Assume that $\sum_{i=1}^r d_i^2 \le n+1$.
Then $X(F)/R = 1$, i.e., there is a unique $R$-equivalence class.
\end{theo}  
Here is rough sketch;  see \cite{Pirutka} for details.
Suppose that $\cX \ra B$ is a regular projective model and $s_1,s_2:B \ra \cX$
are two sections corresponding to $x_1,x_2 \in X(F)$.  For simplicity, 
we assume these points are general in $X$.   The evaluation map 
$$\ev:\ocM'_{0,2}(X,m) \ra X\times X$$
is dominant with rationally connected fibers for $m\ge 2$.  We also have a relative evaluation map
$$\ev_{/B}:\ocM'_{0,2}(\cX,m\beta) \ra \cX \times_B \cX$$ 
where $m\beta$ is the class of degree $m$ fibral curves.  Consider the pre-image
$\cP=\ev_{/B}^{-1}(s_1,s_2)$, which is rationally connected (since $s_1$ and $s_2$ are general)
and thus itself admits a section
$\rho:B \ra \cP$.  Indeed, such sections are Zariski-dense in the fibers, and thus correspond
to freely immersed rational curves in the generic fiber $f: \bP^1 \ra X$.  This rational
curve joins $x_1$ and $x_2$, proving direct $R$-equivalence.  
If $x_1$ and $x_2$ are not general, a more sophisticated argument (via specialization) still
yields a {\em chain} of rational curves joining $x_1$ and $x_2$, which suffices to prove
their $R$-equivalence.

\section{Questions for further study}
\label{sect:questions}
\begin{enumerate}
\item{A. Corti \cite{Cor} has developed a theory of `standard models' for del Pezzo 
surfaces of degree $\ge 2$ over Dedekind schemes;  he offers a fairly explicit
description of the singularities that arise.  Can weak approximation be established
using these standard models?  }
\item{While weak and strong approximation coincide for proper varieties,
they can differ in general.  What can we say about these properties for open log Fano
varieties $Y$ over $F=k(B)$, e.g., $Y=X\setminus D$ where $X$ is smooth and proper, $D\subset X$
is a reduced normal crossings divisor, and $-(K_X+D)$ is ample? }
\item{And what about integral points of open log Fano varieties?  
Suppose $X$ and $D$ are as above with compatible models
$\pi:(\cX,\cD) \ra B.$
Given a finite set $S \subset B$, we can consider $S$-integral points, defined as sections
$\sigma:B\ra \cX$ such that $\sigma^{-1}(\cD) \subset S$ set-theoretically.  What approximation
properties do these satisfy?  (See \cite{HT08b} for Zariski density results in this context.)}
\item{In the proofs of Theorem~\ref{theo:WARSC1} and \ref{theo:WARSC2},
to what extent can the technical assumptions on rational simple connectedness be weakened?  
Does it suffice to exhibit some curve class $\beta$ such that 
$$\ev:\cM_{0,2}(X,\beta) \ra X \times X$$
is dominant with rationally-connected fibers?}
\end{enumerate}

\appendix
\section{Appendix: Stable maps}
We continue to work over an algebraically closed field $k$ of characteristic zero.

\paragraph{What is a stable map?}
We review some basic terminology:  
A {\em curve } $C$ is a reduced connected projective scheme
of pure dimension one over $k$.  The {\em genus} of a curve
is its arithmetic genus $1-\chi(\cO_C)$.  A point $p\in C$ is a {\em node} if
either of the following equivalent conditions is satisfied
\begin{itemize}
\item{the tangent cone of $C$ at $p$ is isomorphic to
$$\Spec(k[x,y]/\left<xy\right>);$$}
\item{$C$ has two smooth branches at $p$, meeting transversally.}
\end{itemize}
A {\em marked point} of $C$ is a smooth point $s\in C$ and 
a {\em prestable curve with $n$ marked points} $(C,s_1,\ldots,s_n)$ is a 
nodal curve $C$ with $s_1,\ldots,s_n \in C$ such that $s_i\neq s_j$
when $i\neq j$.  

If $C$ is a nodal curve of genus zero then each irreducible component of $C$
is isomorphic to $\bP^1$.

Let $C$ be a nodal curve with normalization $\nu:C^{\nu} \ra C$
with connected components $C^{\nu}_1,\ldots,C^{\nu}_m$
and images $C_1=\nu(C^{\nu}_1),\ldots,C_m=\nu(C^{\nu}_m)$.
The {\em dual graph}
of $C$ is a graph
with vertices $\{v_i, i=1,\ldots,m\}$ and edges $\{e_p\}$ indexed by the nodes
$p\in C$, i.e., given $p \in C$ a node with $\nu^{-1}(p)=\{p',p''\}$ where
$p' \in C^{\nu}_i$ and $p'' \in C^{\nu}_j$, then the edge $e_p$ joins $v_i$ to $v_j$.  
We often put additional structure on this graph, e.g., we can label the vertex
$v_i$ with $g_i=\mathrm{genus}(C^{\nu}_i)$.  
For a curve with marked point we add a tail
(i.e., an edge with just one endpoint) at the vertex corresponding to the 
component of the normalization containing the point, i.e., if $s_i$
lies on $C_j^{\nu}$ then we attach a tail $t_i$ to the vertex $v_j$.

Let $C$ be a nodal curve with dualizing sheaf $\omega_C$
and $C^{\nu}_i$ a component of its normalization.  Then
$$(\nu^*\omega_C)|C^{\nu}_i=\omega_{C^{\nu}_i}(D_i)$$
where $D_i=\sum_{\text{nodes } p \in C_i} \nu^{-1}(p) \cap C^{\nu}_i$.

\begin{defi}
Let $(C,p_1,\ldots,p_n)$ be a nodal curve 
with $n$ marked points.  It is {\em stable} 
if $\omega_C(p_1+\cdots+p_n)$ is ample.  
In combinatorial terms, if $v_i$
is a vertex of the dual graph with $g_i=0$ (resp.~$g_i=1$) then at least
three (resp.~one) edges or tails are incident to $v_i$.

Now fix a scheme $X$.  
A morphism $f:(C,s_1,\ldots,s_n)\ra X$ is called a {\em prestable map};
it is {\em stable} if $\omega_C$ is ample relative to $f$, i.e., its
restriction to each irreducible component of $C$ contracted by $f$ has
positive degree.  
Combinatorially, if $v_i$
corresponds to a contracted component
and $g_i=0$ (resp.~$g_i=1$) then at least
three (resp.~one) edges or tails are incident to $v_i$.

Two prestable maps with $n$ marked points
$$f:(C,s_1,\ldots,s_n)\ra X, \quad 
f':(C',s'_1,\ldots,s'_n)\ra X$$ 
are {\em isomorphic} if there 
is an isomorphism $\iota:C \ra C'$ over $X$ with
$\iota(s_i)=s'_i$ for each $i$.  
\end{defi}
Note that precomposing a stable map with an automorphism of the curve
gives an isomorphic stable map.  

Given a morphism of schemes
$\phi:X \ra B$ and a prestable map
$f:(C,s_1,\ldots,s_n) \ra X$, postcomposing by $\phi$ clearly gives a 
prestable map
$$\phi_*f:=\phi \circ f:(C,s_1,\ldots,s_n) \ra B.$$

\begin{defi}
Let $S$ be a scheme of finite type over $k$.  A {\em family of 
prestable curves with $n$ marked points} consists of a flat
proper morphism $\pi:\cC \ra S$ and sections 
$s_1,\ldots,s_n$ of $\pi$ such that each geometric fiber 
$$(\cC_t=\pi^{-1}(t),s_1(t),\ldots,s_n(t)), \quad t\in S(k)$$
is prestable with $n$ marked points.  

Given a scheme $\cX \ra S$, a {\em prestable map}
consists of a family of prestable curves with marked points
$(\cC,s_1,\ldots,s_n) \ra S$
and a morphism:
$$\begin{array}{rcccl}
\cC & & \stackrel{f}{\ra} & &\cX \\
    &\searrow & & \swarrow & \\
    &         & S & & 
\end{array}
$$
It is a {\em stable map} if, for each $t\in S(k)$, the induced
$$f_t:(\cC_t,s_1(t),\ldots,s_n(t)) \ra \cX_t$$ 
is a stable map.  
\end{defi}
Usually we have $\cX=X \times S$ for some scheme $X$ over $k$,
in which case we call these families of prestable/stable
maps to $X$.  

Again, given a morphism of schemes over $S$
$$\begin{array}{rcccl}
\cX & & \stackrel{\phi}{\ra} & &\cB \\
    &\searrow & & \swarrow & \\
    &         & S & & 
\end{array}
$$
and a prestable map
$f:(\cC,s_1,\ldots,s_n) \ra \cX$, the composition
$$\phi_*f=\phi \circ f:(\cC,s_1,\ldots,s_n) \ra \cB$$
is also prestable.  

\paragraph{Statement of results}
The first two results are fairly standard applications of the language
of algebraic stacks:

\begin{theo}[Existence I]
\label{theo:existI}
Let $X$ be a scheme over $k$.  
There exists a (non-separated) Artin
stack $\cM^{ps}_{g,n}(X)$, locally of finite type over $k$, representing
families of prestable maps of genus $g$ curves with $n$ marked points
into $X$.  Given a morphism $\phi:X \ra B$ over $k$,
postcomposition induces a morphism (i.e., a $1$-morphism)
$$\phi_*:\cM^{ps}_{g,n}(X) \ra
\cM^{ps}_{g,n}(B).$$
\end{theo}

\begin{theo}[Existence II]  \label{theo:existII}
Let $X$ be a proper scheme over $k$.  There exists 
an open substack
$$\ocM_{g,n}(X) \subset \cM^{ps}_{g,n}(X)$$
parametrizing stable maps.  Each connected component
is a proper Deligne-Mumford stack 
of finite type over $k$.  If $X$ is projective
then each connected
component of the coarse moduli space is projective.
\end{theo}

The third statement uses slightly different techniques:
\begin{theo}[Stabilization]
Let $X$ be a scheme over $k$.  Consider the open
substack
$$\cM^{\circ}_{g,n}(X) \subset \cM^{ps}_{g,n}(X)$$
parametrizing maps that are non-constant or satisfy $2g-2+n>0$.  
There exists a stabilization morphism
$$\begin{array}{rcl}
\sigma:\cM^{\circ}_{g,n}(X) &\ra& \ocM_{g,n}(X) \\
\{ f:(C,s_1,\ldots,s_n)\ra X \} & \mapsto & \{ f':(C',s'_1,\ldots,s'_n)\ra X \}
\end{array}
$$
characterized as follows:  To obtain $C'$, successively contract 
each irreducible component $D \subset C$
such that $f(D)=\{\text{point}\}$ and $\omega_C(s_1+\ldots+s_n)|D$ fails to 
be ample.  The morphism $f$ descends to a morphism $f':C'\ra X$.  
\end{theo}

\paragraph{On existence}
We focus on Theorem~\ref{theo:existI}, as the properness/projectivity
assertions of Theorem~\ref{theo:existII} are covered in several places in the literature.  

We sketch a construction for the relevant stack locally in a neighborhood of
$$f_0:(C_0,s_1(0),\ldots,s_n(0))\ra X.$$
We abuse terminology, using the term `Hilbert scheme' for the
algebraic space parametrizing proper subschemes of a scheme, as constructed in
\cite{Art69}.  

First, choose a line bundle $L_0$ on $C_0$ that is very ample with 
no higher cohomology.  Fix a basis for $\Gamma(C_0,L_0)$ and consider
the corresponding embedding 
$C_0\hookrightarrow \bP^N.$  
Let $\Hilb_1$ denote the connected component of the Hilbert scheme
parametrizing nested subschemes
$$\{s_1(0),\ldots,s_n(0) \} \subset C_0 \subset \bP^N.$$
Let $\Hilb_2$ denote the component parametrizing pairs of
nested subschemes
$$\begin{array}{c}
\{s_1(0),\ldots,s_n(0) \} \subset C_0 \subset \bP^N \\
\Gamma_{f_0} \subset C_0 \times X
\end{array}$$
where $\Gamma_{f_0}$ is the graph of our stable map.
Let $0\in \Hilb_2$ denote the distinguished point corresponding
to our choice of prestable map and projective embedding of $C$.  
Restricting to a suitable open neighborhood $0 \in U \subset \Hilb_2$,
we obtain a `universal' prestable map over $U$.  Precisely,
restrict to the $u \in \Hilb_2$ where 
\begin{itemize}
\item{$s_1(u),\ldots,s_n(u)$ are distinct;}
\item{$\cC_u$ is nodal;}
\item{$s_i(u) \in \cC_u$ is a smooth point;}
\item{$\Gamma_u \subset \cC_u \times X$ is the graph of a function.}
\end{itemize}
The relevant prestable map is induced by projection onto $X$.  

{\bf Claim:}  
Let $(S,0)$ denote a pointed scheme and
$$f:(\cC,s_1,\ldots,s_n) \ra X\times S$$
a family of prestable maps agreeing with our initial
stable map at the distinguished point $0$.  
There exists an \'etale neighborhood $S' \ra S$
of $0$ and a morphism
$$\mu:S' \ra U \subset \Hilb_2$$
such that the pull-back of the universal stable map over $U$
is isomorphic to the pull back of our family of prestable maps
$$f':(\cC',s_1',\ldots,s_n') \ra X \times S',$$
where $\cC'=\cC\times_S S'$, the $s'_j$ are the induced sections,
and $f'$ is the morphism obtained by composition $f$ with the
projection $\cC' \ra \cC$.  

Here is the idea of the proof:  We choose $S'$ in such a way that
$\pi':\cC'\ra S'$ admits a relatively very ample 
line bundle $\cL$ restricting
to $L_0$ at $C_0$, and $\pi'_*\cL$ admits a trivialization over $S'$
restricting to our choice of basis of $\Gamma(C_0,L_0)$.  
Using this data, we get a canonical lift of our prestable map to
$U$.  

Each such $U$ gives a open neighborhood for
$$[f_0:(C_0,s_1(0),\ldots,s_n(0))\ra X] \in \cM^{ps}_{g,n}(X)$$
in the smooth topology.  We would like to use these 
as the basis of a stack presentation of $\cM^{ps}_{g,n}$.  
For any two such open sets $U'$ and $U''$, we must describe
the gluing relation between $U'$ and $U''$:  Given universal 
prestable maps
$$\cC' \ra U', \quad f':(\cC',s'_1,\ldots,s'_n) \ra X\times U',$$ 
and 
$$\cC'' \ra U'', \quad f'':(\cC'',s''_1,\ldots,s''_n) \ra X\times U'',$$ 
we glue fibers over $u' \in U'$ and $u'' \in U''$
when there is an isomorphism
$$\iota:\cC'_{u'} \ra \cC''_{u''}$$
such that $f''\circ \phi=f'$.  This is a smooth equivalence relation,
so the quotient is an Artin stack (see \cite[\S 4]{LMB}).  

Suppose $\phi:X\ra B$ is a morphism of projective varieties,
$$f_0:(C_0,s_1(0),\ldots,s_n(0))\ra X$$
a prestable map to $X$, and
$$\phi\circ f_0:(C_0,s_1(0),\ldots,s_n(0))\ra B$$
the induced map to $B$.  Let $\Hilb_2^X$ and $\Hilb_2^B$
denote the Hilbert schemes constructed as above and 
$U \subset \Hilb_2^X$ and $V\subset \Hilb_2^B$ the 
associated open subsets.  Then composition by $\phi$
induces morphisms
$$\begin{array}{rcl}
\phi_*:U & \ra & V  \\
	\Gamma_f & \mapsto & \Gamma_{\phi \circ f}
\end{array}
$$
compatible with the equivalence relations.  Thus we get a morphism of 
the corresponding stacks.

\paragraph{Discussion of stabilization}
Suppose we have a morphism $S \ra \cM^{\circ}_{g,n}(X)$ corresponding
to a prestable map, consisting of a family of prestable curves
$$\pi:(\cC,s_1,\ldots,s_n) \ra S$$
and a morphism
$f:\cC \ra X\times S$
over $S$.  Fix a line bundle $\cL$ on $\cC$ ample relative to $\pi$
and thus ample relative to $f$.  

Suppose that $\omega_{\pi}(s_1+\cdots+s_n)$ fails to be ample relative
to $f$, i.e., some fiber $\cC_t$ admits an irreducible component $D$
such that $f(D)=\{\text{point}\}$ and $\omega_{\pi}(s_1+\cdots+s_n)$
has negative degree along $D$.  It is evident that there are
a finite number of curve classes with this property.  Choose $c>0$ 
to be the smallest positive number such that the $\bQ$-divisor
$$\omega_{\pi}(s_1+\cdots+s_n) \otimes \cL^{c}$$
is nef relative to $f$.  Write $c=a/b$ with $a,b \in \bN$ so that
$$M=\omega^b_{\pi}(b(s_1+\cdots+s_n)) \otimes \cL^a$$
is Cartier.  Note that $M$ is {\em trivial} along any component
$D \subset \cC_t$ over which it has zero degree.  

One can show that for some $N>0$, $M^N$ is globally generated and has no 
higher direct images relative to $f$.  (We give references below.)
In particular $\oplus_{N\ge 0}f_*M^N$ is finitely generated and
we get a morphism 
$$\begin{array}{rcccl}
\cC & & \stackrel{\beta}{\lra} & & \cC':=\Proj(\oplus_{N\ge 0} f_*M^N)\\ 
    & \searrow & & \swarrow & \\
    &          & X\times S & & 
\end{array}
$$
contracting all the components over which $M$ is trivial.  
In some sense, this is a strong version
of the Kawamata
basepoint freeness theorem relative to $f$.  

\begin{rema}
A more direct argument can be made when $X$ is projective \cite[\S 3]{BM}.  
Choose an ample divisor $H$ on $X$ and consider
$$M=\omega_{\pi}(p_1+\cdots+p_n) \otimes f^*\cO_X(3H).$$
Now $M$ generally has base points, e.g., along components
$D \subset \cC_t$ contracted by $f$ that are isomorphic to $\bP^1$ 
and have just one distinguished point (marked point or node).  
Nevertheless, one can show there is still a morphism
$$\beta:\cC \ra \cC':=\Proj(\oplus_{N\ge 0} \pi_*M^N).$$
\end{rema}

\paragraph{Key literature}
The concept of stable maps goes back to Kontsevich \cite{K,KM}.
A good algebro-geometric introduction can be found in \cite{FP};  
this survey gives a complete construction when the target space is projective.
A stack-theoretic discussion of stable maps can be found in \cite{AV}:
\cite[\S 5]{AV} addresses existence of the moduli space
of stable maps into a projective variety;
\cite[\S 6]{AV} discusses
properness and projectivity of the coarse moduli space;
the generalization to the case of maps into a proper scheme or algebraic
space can be found in \cite[8.4]{AV}.  
An analysis of the stabilization functor---with
applications to functoriality of stable maps---can be found in 
\cite[\S 3]{BM}.  
Explicit constructions of stabilization morphisms in related contexts can be found in \cite{KnMu}
and \cite[\S 3.1,4.1]{Ha};  these references prove the relative global generation and vanishing asserted above.

\bibliographystyle{alpha}
\bibliography{Strasbourg}

\noindent \textbf{Department of Mathematics, Rice University, Houston, Texas 77005,
USA,} 
\texttt{hassett@rice.edu}

\end{document}